\documentclass[a4paper,12pt]{amsart} 
\usepackage[centertags]{amsmath}
\usepackage{amsbsy}
\usepackage{amssymb}
\usepackage{color}
\usepackage{t1enc}
\usepackage[utf8]{inputenc}
\usepackage[mathscr]{eucal}
\usepackage[margin=2cm]{geometry}
\newtheorem{Theorem}{Theorem}
\newtheorem{Lemma}{Lemma}
\newtheorem{Corollary}{Corollary}
\newtheorem{Remark}{Remark}

\newtheorem{Example}{Example}

\newcommand{\eq}{\mathop{\Longleftrightarrow}\limits}
\newcommand{\imp}{\mathop{\Longrightarrow}\limits}
\newcommand{\cal}{\mathcal}

\def \NN{\mathbb N}
\def \QQ{\mathbb Q}
\def \RR{\mathbb R}
\def \SS{\mathbb S}
\def \RV{\mathcal {RV}}
\def \Uc{\mathcal U(c_0(x))}
\def \Uq{\mathcal U(x^q)}
\def \Ic{\mathcal I(c_0(x))}
\def \Ixq{\mathcal I(x^q)}
\def \Iq{\mathcal I_{<q}}

\def \Ieq{\mathcal I_{\leq q}}
\def \I0{\mathcal I_{0}}
\def \Is{\mathcal I}
\def \Fin{\mathcal Fin}
\def \bs{\backslash}

\begin{document}
\today
\title {Characteristics of distributions of sets and their $(R)$- and $(N)$-denseness}

\author[P. Miska]{Piotr Miska}

\address{Institute of Mathematics 
	\\Faculty of Mathematics and Computer Science
	\\Jagiellonian University in Krak\'ow
	\\ul. {\L}ojasiewicza 6, 30-348 Krak\'ow
	\\Poland\\
}
\email{piotr.miska@uj.edu.pl}

\author[J. T. T\'oth]{J\'anos T. T\'oth}

\address{
	Department of Mathematics\\
	J. Selye University\\
	P. O. Box 54\\945 01 Kom\'arno\\
	Slovakia\\
}
\email{tothj@ujs.sk}

\begin{abstract} Let $0\leq q\leq1$ and $\NN$ denotes the set of all positive integers. In this paper we will deal with it too the family $\Uq$ of all regularly distributed set  $X \subset \NN$ whose ratio block sequence of type \eqref{1} is asymptotically distributed with distribution function $g(x) = x^q;\ x \in(0,1]$, and we will show that the regular distributed set, regular sequences, regular variation at infinity are equivalent notations. In this paper also we discuss the relation ship between notations as (N)-denseness, directions sets, generalized ratio sets, dispersion of sequence \eqref{1} and exponent of convergence.
\end{abstract}

\maketitle

\parindent=0pt

\section*{Declarations}
The research of the first author is supported by the Grant of the Polish National Science Centre No. UMO-2019/34/E/ST1/00094. The research of the second author is supported by The Slovak Research and Development Agency under the grant VEGA No. 1/0776/21.

There are neither conflicts of interest nor competing interests.

Our manuscript has no associated data.

 \section{Introduction}
 
 In the whole paper we assume $X=\{x_1<x_2<\cdots <x_n<\cdots\}\subset\NN$ where $\NN$ denotes the set of all positive integers.
 
 The following sequence derived from $X$
 \begin{equation}  
 \frac{x_1}{x_1}, \frac{x_1}{x_2}, \frac{x_2}{x_2},
 \frac{x_1}{x_3}, \frac{x_2}{x_3}, \frac{x_3}{x_3}, \dots ,
 \frac{x_1}{x_n}, \frac{x_2}{x_n}, \dots , \frac{x_n}{x_n}, \dots
 \label{1}
 \end{equation}
 is called \emph{the ratio block sequence} of the set $X$. It is formed by the blocks $X_1, X_2, \dots , X_n , \dots $ where
 $$
 X_n = \left( \frac{x_1}{x_n}, \frac{x_2}{x_n}, \dots , \frac{x_n}{x_n}\right) ,\quad n = 1,2,\dots
 $$
 is called the $n$\emph{-th block}. This kind of block sequences was introduced by O.~Strauch and J.~T.~T\'oth
 \cite{ST} and they studied the set $G(X_n)$ of its distribution functions.
 \\Further, we will be interested in ratio block sequences of type~\eqref{1} possessing an asymptotic distribution function, i.e. $G(X_n)$ is a singleton (see definitions in the next section). 
 
 By means of these distribution functions in \cite{TBFM} there was defined some families of subsets of $\NN$. For $0\leq q\leq 1$ we denote by $\Uq$ the family of all regularly distributed set $X \subset \NN$ whose ratio block sequence is asymptotically distributed with distribution function $g(x) = x^q;\ x \in(0,1]$.
 Further in \cite{TBFM} the following interesting results can be seen, that the exponent of convergence $\lambda$ is closely related to distributional properties of sets of positive integers. More precisely, for each $q \in [0,1]$ the family $\Ieq$ of all sets $A \subset \NN$ such that $\lambda(A) \leq q$ is identical with the family $\Ixq$ of all sets $A \subset \NN$ which are covered by some regularly distributed set $X \in\Uq$.
 \\The \emph{exponent of convergence} of a set $A\subset\NN$ is defined by
 $$
 \lambda(A) = \inf \{ s \in (0,\infty):\ \sum\limits_{n \in \NN} a_n^{-s} < \infty \},
 $$     
 where $A = \{ a_1 < a_2 < \cdots \} \subset \NN$.
 \\In this paper  we will show that the regular distributed set, regular sequences, regular variation at infinity are equivalent notations and also we discuss the relationship between notations as (N)-denseness, directions sets, generalized ratio sets, dispersion of sequence \eqref{1} and exponent of convergence.
 \\The rest of our paper is organized as follows. In Section 2 and Section 3  we recall some known definitions, notations and theorems, which will be used and extended. In Section 4 our new results are presented. Section 5 summarizes the results in chains of implications.

 \section{Definitions}
The following basic definitions are from  papers \cite{PS, OS, ST, S1, S2, TFBZ, TBFM}.

\begin{itemize}\itemsep2pt
	\item[{\tiny$\bullet$}]  $1\le x_1<x_2<\cdots$ denotes a sequence of positive integers.
	
	\item[{\tiny$\bullet$}] For each $n\in \NN$ consider the \emph{step distribution function}
	\[
	F(X_n,x)=\frac{\#\{i\le n: \frac{x_i}{x_n}<x\}}n\,,
	\]
	for $x\in[0,1)$, and for $x=1$ we define $F(X_n,1)=1$.
	
	\item[{\tiny$\bullet$}] A non-decreasing function $g : [0, 1] \to [0, 1]$, $g(0) = 0$, $g(1) = 1$ is called a
	\emph{distribution function} (abbreviated d.f.). We shall identify any two d.f.s
	coinciding at common points of continuity.
	
	\item[{\tiny$\bullet$}] A d.f. $g(x)$ is a d.f. of the sequence of blocks $X_n$, $n = 1, 2, \ldots$, if there
	exists an increasing sequence $n_1 < n_2 < \ldots$ of positive integers such that
	\[
	\lim_{k\to\infty} F(X_{n_k},x)=g(x)
	\]
	a.e. on $[0, 1]$. This is equivalent to the weak convergence, i.e., the preceding
	limit holds for every point $x \in [0, 1]$ of continuity of $g(x)$.
	
	\item[{\tiny$\bullet$}] Denote by $G(X_n)$ the set of all d.f.s of $X_n$, $n = 1, 2, \ldots$. The set of  distribution functions of ratio block sequences was  studied in \cite{BMST2, BMST1, BFT1, BT1, BU, FMT1, GS, KN, LS, PS}.  
	\newline
	If $G(X_n) =\{g(x)\}$
	is a singleton, the d.f. $g(x)$ is also called the \emph{asymptotic distribution function} (abbreviated
	a.d.f.) of $X_n$.
	\newline
	Especially, if $G(X_n)=\{x\}$, then we say that the
	sequence of blocks $X_n$ is uniformly distributed (abbreviated as u.d.) in $[0,1]$.
	
	\smallskip

\item[{\tiny$\bullet$}] The function $\lambda:2^\NN\rightarrow[0,1]$ defined by	
$$
	\lambda(A)=\inf\Big\{t>0:\sum_{a\in A}\frac{1}{a^t}<\infty\Big\}
$$
is called the \emph{exponent of convergence} of a set $A\subset\NN$.
	 
If $q>\lambda(A)$, then $\sum_{a\in A}\frac{1}{a^q}<\infty$ and if  $q<\lambda(A)$, then $\sum_{a\in A}\frac{1}{a^q}=\infty$. In the case when  $q=\lambda(A)$, the series $\sum_{a\in A}\frac{1}{a^q}$ can be either convergent or divergent. 
\\From (\cite[p.26, Exercises 113, 114]{PS}) it follows that the set of all possible values of $\lambda$ forms the whole interval $[0,1]$, i.e. $\{\lambda(A):A\subset\NN\}=[0,1]$ and if $A=\{a_1<a_2<\cdots<a_n<\cdots\}$ then $\lambda(A)$ can be calculated by
$$
\lambda(A)=\limsup_{n\to\infty}\frac{\log n}{\log a_n}.
$$
Here and in the whole paper we use $\log$ for the natural logarithm.\\
Evidently the exponent of convergence $\lambda$ is a monotone set function, i.e. $\lambda(A)\leq\lambda(B)$ for $A\subset B\subset\NN$ and also $\lambda(A\cup B)=\max\{\lambda(A),\lambda(B)\}$ holds for all $A,B\subset\NN$. 

\item[{\tiny$\bullet$}] 

For every $f:\NN\to [0,\infty)$ such that $\sum_{n=1}^\infty f(n)=\infty$ we define a summable ideal generated by the function $f$ by
\[
\Is_f=\Big\{B\subset\NN:\ \sum_{n\in B} f(n)<\infty \Big\} .
\]
In particular, if $f(n)=1/n^q$ with $0\leq q\leq 1$ we obtain the ideal 
\[
\Is_{1/n^q}=\Big\{B\subset\NN:\ \sum_{n\in B} \frac{1}{n^q}<\infty \Big\},
\]
where $\Is_{1/n^0}=\Fin=\{A\subset\NN:\ A \ \textrm{is finite} \}$.

By means of $\lambda$ we define the following ideals (see \cite{TFBZ}):
\\$\Iq=\{A\subset\NN:\lambda(A)<q\} $ for $0<q\leq1$,
\\$\Ieq=\{A\subset\NN:\lambda(A)\leq q\} $ for $0\leq q\leq1$ and
\\$\I0=\{A\subset\NN:\lambda(A)=0\} $.
\\Obviously $\mathcal I_{\leq 0}=\I0$ and $\mathcal I_{\leq 1}=2^\NN$. Moreover, we introduce the following families of subsets of $\NN$ (see \cite{TBFM}):
	$$\Uc = \{X\subset\NN:G(X_n)=\{c_0 (x)\}\}, $$  
	$$\Ic=\{A\subset \NN : \exists X\in\Uc, A\subset X\}, $$
	\\and for $0<q\leq 1$
	$$\Uq =\{X\subset\NN:G(X_n)=\{x^q\}\}, $$
	$$\Ixq= \{A\subset \NN : \exists X\in\Uq, A\subset X\}, $$
	where
	\[
	c_0(x) = \left.
	\begin{cases}
	0 & \text{if } x = 0, \\
	1 & \text{if } 0<x\leq 1.
	\end{cases}\right.
	\]
	Obviously, $$\Uc\subsetneq\Ic ,\quad \Uq\subsetneq\Ixq .$$ 

For a finite set $A\subset\NN$ we have $\lambda(A)=0$. Consequently, $\Fin\subset\I0$. The families $\Iq, \Ieq$ and the well known family $\Is_{1/n^q}$ are related for $0<q<q'<1$ by following inclusions (see \cite{TFBZ}, Th.1.)
\begin{equation}\label{2}
\Fin\subsetneq\I0\subsetneq\Iq\subsetneq\Is_{1/n^q}\subsetneq \Ieq\subsetneq \mathcal I_{<q'}\subsetneq\mathcal I_{<1},
\end{equation}
and the difference of successive sets is infinite, so equality does not hold in any of the inclusions.
\\In this paper we deal with such sets Y which are from some family $\Uq$ for $0<q<1$, so the asymptotic density of the set Y is zero ($d(Y)=0$), since $$\Uq\subsetneq\Ixq=\Ieq\subset\Is_{<1}\subset\Is_{1/n^1}\subset\Is_d=\{A\subset\NN: d(A)=0\}\ .$$ 
The standard way to evaluate the size of sets of real numbers is to take their Lebesgue measure. In this theory sets of Lebesgue measure zero are the smallest ones and often considered as negligible. Sometimes it is important to distinguish also among these sets. To do so, the best way is to use the Hausdorff dimension.
The most frequently used characteristics to evaluate the size of sets of positive integers is the (upper or lower) asymptotic density. Here again, sets of asymptotic density zero are sometimes considered as negligible. In this case, to distinguish among them, the best way is to use the exponent of convergence, or equivalently, the exponential density.

\item[{\tiny$\bullet$}] 
Let $\mathcal I\subset 2^\NN$. Then $\mathcal I$ is called an \emph{admissible ideal} of subsets of positive integers, if  $\mathcal I$ is additive (if $A,B\in\mathcal I$ then $A\cup B\in\mathcal I $), hereditary (if $A\in\mathcal I$ and $B\subset A$ then $B\in\mathcal I$), containing all finite subsets of $\NN$ and it does not contain $\NN $.

\item[{\tiny$\bullet$}] Suppose $f(x)$ is eventually defined on $(1,\infty)$ and eventually positive. Given $\alpha\in \RR$, $f(x)$ has regular variation at infinity of index $\alpha$, abbreviated as $f(x)\in\RV_{\alpha}$, if 
	$$\lim_{x\to\infty}\frac{f(cx)}{f(x)}=c^{\alpha},\ \textrm{ for every}\ c>0 .$$
(see \cite{BU}, p.\,136).	If $f(x)\in \RV_0$ then $f(x)$ is also said to be slowly varying at infinity.

\item[{\tiny$\bullet$}]
A set $A=\{a_1<a_2<\cdots<a_n<\cdots\}\subset\NN$ satisfying
$$A(x)\sim x^q\phi(x),$$
where fuction $\phi(x)\in\RV_0$ and $0<q\leq 1$,
will be called a regular sequence with exponent $q$ (see \cite{PS}, Exerc.\,153, p.\,86).

\item[{\tiny$\bullet$}] The concept of a ratio set has been introduced in the papers \cite{S1}, \cite{S2}. If $A\subset\NN$, $B\subset\NN$, then we put $R(A,B)=\{\frac ab;\ a\in A,\  b\in B\}$. The set $R(A,B)$ is said to be the ratio set of $A$ and $B$. In particular, for $A=B$ we put $R(A,A)=R(A)$. The symbol $X^d$ will stand for the set of all accumulation points of $X\subset [0,\infty]$. It is easy to see that for any infinite subsets $A,\ B$ of positive integers $\{0,+\infty\}\subset R(A,B)^d$. The set $R(A,B)$ is dense in $(0,\infty)$ if $R(A,B)^d=[0,+\infty]$. Note here that $R(A,B)\neq R(B,A)$ in general, however $R(A,B)$ is dense in $(0,\infty)$ if and only if $R(B,A)$ is dense in $(0,\infty)$.
\\We say that the set $A\subset\NN$ is $(N)$-dense set if $R(A,B)$ is dense set in $(0,\infty)$ for arbitrary infinite set $B\subset\NN$ and A is $(R)$-dense set if $R(A)$ is dense set in $(0,\infty)$. Therefore $(N)$-denseness implies $(R)$-denseness.

\item[{\tiny$\bullet$}] Let us denote
\begin{align*}
\RR_+=&(0,\infty),\\
\SS^k=&\{(x_1,\ldots,x_k,x_{k+1})\in\RR^{k+1}: \sum_{i=1}^{k+1} x_i^2=1\},\\
\SS_+^k=&\SS^k\cap\RR_+^{k+1}
\end{align*}
for each $k\in\NN$. Then, for each $k\in\NN$, $k\geq 2$, we define $k-1$-dimensional ratio set of given sets $A_1,\ldots,A_{k-1},B\subset\NN$ as
$$R^k(A_1,\ldots,A_{k-1};B)=\left\{\left(\frac{a_1}{b},\ldots,\frac{a_{k-1}}{b}\right): a_1\in A_1,\ldots,a_{k-1}\in A_{k-1},b\in B\right\}.$$
If $A_1=\ldots =A_{k-1}=A$, then we will write $R^k(A;B)$ instead of $R^k(A,\ldots ,A;B)$. Similarly, if $A_1=\ldots =A_{k-1}=B=A$, then we shall write $R^k(A)$ instead of $R^k(A;A)$. Of course, $R^2(A;B)=R(A,B)$ is the well known ratio set of $A$ and $B$. In particular, $R^2(A)=R(A)$. Moreover, let us introduce for each $k\in\NN$, $k\geq 2$, the $k$-th directions set of sets $A_1,\ldots,A_{k-1},B\subset\NN$ as 
$$D^k(A_1,\ldots,A_{k-1};B)=\left\{\left(\frac{a_1}{\sqrt{\sum_{i=1}^k a_i^2}},\ldots,\frac{a_k}{\sqrt{\sum_{i=1}^k a_i^2}}\right): a_1\in A_1,\ldots,a_{k-1}\in A_{k-1}, a_k\in B\right\}.$$
If $A_1=\ldots =A_{k-1}=A$, then we will write $D^k(A;B)$ instead of $D^k(A,\ldots ,A;B)$. Similarly, if $A_1=\ldots =A_{k-1}=B=A$, then we shall write $D^k(A)$ instead of $D^k(A;A)$.
The concept of directions sets $D^k(A)$ as generalizations of ratio sets was introduced and studied by P. Leonetti and C. Sanna in \cite{LS}.

\item[{\tiny$\bullet$}] Let us define dispersion of an infinite set $A=\{a_1<a_2<a_3<\ldots\}\subset\NN$ as 
$$\underline D (A)=\liminf_{n\to\infty}\frac{1}{a_n}\max\{a_1,a_{i+1}-a_i:i\in\{1,\ldots,n-1\}\}.$$
Let us notice that dispersion is an non-increasing function on the set $2^{\NN}\setminus\Fin$ of all infinite subsets of $\NN$, i.e. if $A\subset B\subset\NN$ are infinite, then $\underline{D}(A)\geq\underline{D}(B)$. Moreover, $\underline{D}(\NN)=0$ and every value from the interval $[0,1]$ is attained as dispersion of some subset of $\NN$ (see \cite{TMF}).
\end{itemize}

\section{Overview of the known results}

In this section we mention well known results related to the topic of this paper and some other ones we use in the proofs of our theorems. In the whole section in (A1)--(A8) we assume $X=\{x_1<x_2<\cdots <x_n<\cdots\}\subset\NN$.

\medskip

\begin{itemize}\itemsep2pt

	\item[(A1)] Assume that $G(X_n)$ is singleton, i.e., $G(X_n) =\{g(x)\}$. Then either $g(x) =c_0(x)$ for $x \in [0, 1]$; or $g(x) = x^q$ for $x \in [0, 1]$ and some fixed $0 < q \le 1$ \cite[Th. 8.2]{ST}.
\end{itemize}

	The sets $X=\{x_1<x_2<\cdots\}$ from $\Uc$ are characterized by (A4) and these ones belonging to $\Uq$ are characterized by (A2) and (A6). In (\cite[Theorem 1 and Example 1]{TBFM}) there is proved that the family $\Uc$ is additive, i.e. it is closed with respect to finite unions and does not form an ideal as it is not hereditary, i.e. there exists sets $C\in\Uc$ and $B\subset C$ such that $B\notin \Uc$. On the other hand the family $\Ic$ is an ideal (\cite{TBFM}, Theorem 2). For these families the following statements hold.
	
\begin{itemize}\itemsep2pt	
	\item[(A2)] 	Let $0<q\le 1$ be a real number. Then 
	$$X\in\Uq\ \Longleftrightarrow\ \forall\ k\in\NN:\  \lim_{n\to\infty}\frac{x_{kn}}{x_n}=k^{\frac{1}{q}}\ .$$
	\cite[Th. 1]{FT}

	\item[(A3)] 	Let $0<q\le 1$ be a real number. If $X\in\Uq$, then
	$$\lim_{n\to\infty}\frac{x_{n+1}}{x_n}=1.$$
	\cite[Remark 3]{FMT1}

	\item[(A4)] We have
	$$X\in\Uc\ \Longleftrightarrow \ \lim_{n\to\infty} \frac1{nx_n}\sum_{i=1}^n x_i =0.$$
	\cite[Th. 7.1]{ST}

	\item[(A5)] We have
	$$c_0(x) \in  G(X_n)  \Longleftrightarrow
	\liminf_{n\to\infty} \frac1{nx_n}\sum_{i=1}^n x_i =0.$$
	\cite[Th. 4]{FMT1}
	
	\item[(A6)]
	Let $0<q\le 1$ be a real number. Then 
	$$X\in\Uq\ \Longleftrightarrow\ \lim_{n\to\infty} \frac1{nx_n} \sum_{i=1}^{n} {x_i}=\frac q{q+1}.$$
	\cite[Th. 1]{BFT1}
	
	\item[(A7)] 
	If $X\in\Uc$. Then
	$$\lim_{n\to\infty}\frac{\log n}{\log x_n}=0\ (\textrm{i.e.}\ \lambda(X)=0). $$
	\cite[Th. 2]{BFT1}
    
    \item[(A8)] 	Let  $0<q\leq 1$ be a real number. If $X\in\Uq$ then
    $$ \lim_{n\to\infty}\frac{\log n}{\log x_n}=q\ (\textrm{therefore}\ \lambda(X)=q). $$
    \cite[Th. 3]{BFT1}
    
    \item[(A9)] 	Let $0<q\leq 1$. Then each of the families $\I0,\  \Iq$ and $\Ieq$ forms an admissible ideal, except for $\mathcal I_{\leq 1}$ \cite[Th. 1]{TFBZ}.
    
    \item[(A10)] Let $0<q\leq 1$. Then each of the families $\Ic,\  \Ixq$ forms an admissible ideal and $\Ic=\I0$, $\Ixq=\Ieq$ \cite[Th. 5 and Th. 7]{TBFM}.
\end{itemize}

In (A10) there are characterized sets $A\subset\NN$ belonging to ideals $\Ic$ or $\Ixq$ by means of their exponent of convergence, i.e. $\lambda(A)=0$ or $\lambda(A)\leq q$, which means that $A\in\I0$ or $A\in \Ieq$. From (A8) and (A10) we obtain also the following interesting inclusion holding for studied  families of sets (for characterization of $\Ieq$ and $\Iq$ see (A11) and (A12)):
    $$\Uq\subset\Ieq\setminus\Iq ,\ \textrm{which implies }\ \Iq\subset\Ixq\setminus\Uq .$$

\begin{itemize}\itemsep2pt	    
    \item[(A11)] Let $0\leq q<1$ be real and define the counting function of $A\subset \NN$ as
    $A(x)=\#\{a\leq x: a\in A \}$ for $x\geq 1$. Then
    $$A\in\Ieq\ \Longleftrightarrow\ \forall\ \delta>0:\ \lim_{x\to\infty}\frac{A(x)}{x^{q+\delta}}= 0\ .$$
    \cite[Th. 3]{TFBZ}
    
    \item[(A12)] Let $0<q\leq 1$ be a real number and $A\subset\NN$. Then
    $$A\in\Iq\ \Longleftrightarrow\ \exists\ \delta>0:\  
    \lim_{x\to\infty}\frac{A(x)}{x^{q-\delta}}= 0 .$$
    \cite[Th. 4]{TFBZ}

    \item[(A13)]
    Let  $0<q\leq1$, $ X=\{x_1<x_2<\cdots\}\subset\NN$, $Y=\{y_1<y_2<\cdots\}\subset\NN$, let $g(x)\in\{c_0(x),x^q\}$ be fixed and assume that 
    \[ 
    Y\in\mathcal U(g(x)) \quad\textrm{and } \quad\lim_{t\to\infty}\frac{X(t)}{Y(t)}=0.
    \]
    Then
    $$
    X\cup Y\in\mathcal U(g(x)).
    $$
    \cite[Th. 4]{TBFM}

   \item[(A14)] Let $f(x)$ and $g(x)$ are eventually defined on $\RR$, and eventually positive. Then, for $\alpha, \beta\in\RR$,
   \begin{itemize}
   	\item[i)] 
   	$$f(x)\in\RV_\alpha\ \Longleftrightarrow\ \frac{f(x)}{x^\alpha}\in\RV_0\ .$$
   	(\cite{BU}, Prop. 7. 20, p. 136))
   	\item[ii)] 
   	(a) If $f(x)\sim g(x)\in\RV_\alpha$ then $f(x)\in\RV_\alpha$,
   	\\(b) If $f(x)\in\RV_\alpha$ and $g(x)\in\RV_\beta$ then the reciprocal $\frac{1}{f(x)}\in\RV_{-\alpha}$ , and the product $f(x)g(x)\in\RV_{\alpha+\beta}$,
   	\\(c) $\RV_0$ is closed under multiplication, division, and asymptotic equality.
   	\\(\cite{BU}, Prop. 7. 21, p. 137)
   \end{itemize}  
    
    \item[(A15)] Let $A=\{a_1<a_2<\cdots<a_n<\cdots\}\subset\NN$. Then
    \[
    \lambda(A)=\limsup_{n\to\infty}\frac{\log n}{\log a_n}=\limsup_{x\to\infty}\frac{\log A(x)}{\log x} 
    \]
    and
    \[
    \liminf_{n\to\infty}\frac{\log n}{\log a_n}=\liminf_{x\to\infty}\frac{\log A(x)}{\log x} .
    \]
    \cite[Ex. 148, 149, p. 85]{PS}

    \item[(A16)] Let $A=\{a_1<a_2<\cdots<a_n<\cdots\}\subset\NN$. If $A$ is regular sequence with exponent $q>0$, then $\lambda(A)=q$\ \cite[Ex. 153, p. 86]{PS}.

    \item[(A17)] Let $A=\{a_1<a_2<\cdots<a_n<\cdots\}\subset\NN$. If $A$ is regular sequence with exponent $q>0$, then $A(x)\in\RV_q$ i.e. $\lim_{x\to\infty}\frac{A(cx)}{A(x)}=c^q$ for every $c>0$ \cite[Ex. 154, p. 86]{PS}.

    \item[(A18)] Let $A=\{a_1<a_2<\cdots<a_n<\cdots\}\subset\NN$ be a regular sequence with exponent $q>0$ and let $\alpha>0$. Then
    \[
    \lim_{t\to\infty}\frac{1}{A(t)}\sum_{a_i\leq t}\Big(\frac{a_i}{t}\Big)^{\alpha-q}=\int_{0}^{1} x^{\frac{\alpha-q}{q}}dx=\frac{q}{\alpha}\ .
    \]
    Therefore, as $\alpha=q+1$ we have \[\lim_{n\to\infty} \frac1{na_n} \sum_{i=1}^{n} {a_i}=\frac q{q+1}\ .\]
    \cite[Ex. 157, p. 86]{PS}

    \item[(A19)] Let $\alpha\in[0,\infty)$. Suppose $f(x)\in\RV_{\alpha}$ is eventually nondecreasing and diverges (to infinity). Then
    \[
    \alpha=\lim_{n \to \infty}\frac{\log f(n)}{\log n}\ .
    \]
    \cite[Prop. 7.23., p. 138]{BU}

    \item[(A20)] Let $A\subset\NN$ be a set such that
    \[
    \liminf_{t\to\infty}\frac{A(ct)}{A(t)}>1\ \textrm{for every}\ c>1\ ,
    \]
    then $A$ is $(N)$-dense set. \cite[Satz 3]{S2}

    \item[(A21)] If a set $A=\{a_1<a_2<\cdots<a_n<\cdots\}\subset\NN$ satisfies
    \[
    \lim_{n \to \infty}\frac{a_{n+1}}{a_n}=1\ ,
    \]
    then $A$ is $(N)$-dense set. \cite[Th. 2]{BT1}

    \item[(A22)] Let $A=\{a_1<a_2<\cdots<a_n<\cdots\}\subset\NN$. Then
    \[
    \lim_{n \to \infty}\frac{a_{n+1}}{a_n}=1\ \imp\ \underline D(A)=0\ \imp\ A\ \textrm{is}\ (R)-\textrm{dense set}\ \imp\ \underline D(A)\leq\frac 12\ .
    \]
    \cite[Th. 2, Th. 3]{TMF}
    
    \item[(A23)] Let $A=\{a_1<a_2<\cdots<a_n<\cdots\}\subset\NN$ and $k,l\in\NN$, $2\leq k<l$. Then
    \[
    D^l(A)\ \textrm{is dense in}\ \SS^{l-1}_+\ \imp\ D^k(A)\ \textrm{is dense in}\ \SS^{k-1}_+.\
    \]
    Moreover, there exists a set $A\subset\NN$ such that $D^k(A)$ is dense in $\RR^k_+$ but $D^l(A)$ is not dense in $\RR^l_+$ \cite[Th. 1.4]{LS}.
    
    \item[(A24)] Let $A=\{a_1<a_2<\cdots<a_n<\cdots\}\subset\NN$. Then
    \[
    \lim_{n \to \infty}\frac{a_{n+1}}{a_n}=1\ \imp\ \forall\ k\geq 2:\ D^k(A)\ \textrm{is dense in}\ \SS^{k-1}_+.\
    \]
    \cite[Th. 1.5]{LS}
\end{itemize}

\section{Results}

\begin{Lemma}\label{L1}
Let $A=\{a_1<a_2<\cdots\}\subset\NN$. Then, for every $c>0$ we have
\begin{equation}
	\label{3} 
	\limsup_{t\to\infty}\frac{A(ct)}{A(t)}=\limsup_{n\to\infty}\frac{A(a_n)}{A\big(\frac 1c a_n\big)}\ , 
\end{equation}
\begin{equation}
	\label{4} 
	\liminf_{t\to\infty}\frac{A(ct)}{A(t)}=\liminf_{n\to\infty}\frac{A(a_n)}{A\big(\frac 1c a_n\big)}\ .
\end{equation}
\end{Lemma}

\begin{proof}
	Let $c>0$ be a fixed and denote $l=\limsup_{n\to\infty}\frac{A(a_n)}{A\big(\frac 1c a_n\big)}$. Then, for $\varepsilon>0$	there exists an $n_0\in\NN$ such that for each $n\geq n_0$
	$$\frac{A(a_n)}{A\big(\frac 1c a_n\big)}<l+\varepsilon .$$
	For sufficiently large $t$ there exists $n\geq n_0$ such that $\frac 1c a_n<t\leq \frac 1c a_{n+1}$. Therefore
	\[
	\frac{A(ct)}{A(t)}\leq\frac{A(a_{n+1})}{A\big(\frac 1c a_n\big)}=\frac{A(a_{n+1})}{A(a_n)}\frac{A(a_{n})}{A\big(\frac 1c a_n\big)}<\frac{n+1}{n}(l+\varepsilon) .
	\]
	From this we get (since $n\to\infty$ if $t\to\infty$) that
	$\limsup_{t\to\infty}\frac{A(ct)}{A(t)}\leq l+\varepsilon$ for arbitrary $\varepsilon>0$. Thus 
	$$\limsup_{t\to\infty}\frac{A(ct)}{A(t)}\leq l .$$
	Since on the other hand we obtain (by putting $t=\frac{1}{c}a_n$) that
	$$\limsup_{t\to\infty}\frac{A(ct)}{A(t)}\geq l ,$$
	then there holds equality ~\eqref{3}. The proof of ~\eqref{4} is similar.
\end{proof}

\begin{Theorem}\label{T1}
Let $A=\{a_1<a_2<\cdots\}\subset\NN$ and $g: [0, 1] \to [0, 1]$ be a distribution function of the sequence of blocks $A_n$. Then:
\begin{itemize}
\item[i)] $\liminf\limits_{t\to\infty}\frac{A(ct)}{A(t)}\leq\frac{1}{g\big(\frac 1c\big)}\leq \limsup\limits_{t\to\infty}\frac{A(ct)}{A(t)}$ for each $c>1$ such that $\frac 1c$ is a point of continuity of $g$,
\item[ii)] $\forall\  0<c<1:\ \liminf\limits_{t\to\infty}\frac{A(ct)}{A(t)}\leq g(c)\leq\limsup\limits_{t\to\infty}\frac{A(ct)}{A(t)}$ for each $c\in (0,1)$ being a point of continuity of $g$.
\end{itemize}
Moreover, the conditions i) and ii) are equivalent.
\end{Theorem}

\begin{proof}
Let us start with the proof of i). Let $g(x)\in G(A_n)$ and $c>1$ be its point of continuity. Thus, there exists a sequence $(n_k)$ of positive integers such that  
$$g\Big(\frac 1c\Big)=\lim_{k \to \infty}F\Big(A_{n_k}, \frac 1c \Big)=\lim_{k\to\infty}\frac{\#\big\{i\le n_k: \frac{a_i}{a_{n_k}}<\frac 1c\big\}}{n_k}=\lim_{k \to \infty}\frac{A\big(\frac 1c a_{n_k}\big)}{A(a_{n_k})} . $$
Then 
\[
\liminf_{n \to \infty}\frac{A\big(\frac 1c a_{n}\big)}{A(a_{n})}\leq g\Big(\frac 1c\Big)\leq\limsup_{n \to \infty}\frac{A\big(\frac 1c a_{n}\big)}{A(a_{n})}\ \textrm{i.e.}\ \liminf_{n \to \infty}\frac{A( a_{n})}{A\Big(\frac 1c a_{n}\Big)}\leq \frac{1}{g\big(\frac 1c\big)}\leq\limsup_{n \to \infty}\frac{A(a_{n})}{A\Big(\frac 1c a_{n}\Big)}\ .
\]
From above and Lemma~\ref{L1} we have
$$\liminf_{t\to\infty}\frac{A(ct)}{A(t)}\leq\frac{1}{g\big(\frac 1c\big)}\leq\limsup_{t\to\infty}\frac{A(ct)}{A(t)} . $$

The proof of implication i)$\Rightarrow$ ii) is following. By i) we have
\[
 \liminf_{t\to\infty}\frac{A\big(\frac 1c t\big)}{A(t)}\leq \frac{1}{g(c)}\leq\limsup_{t\to\infty}\frac{A\big(\frac 1c t\big)}{A(t)}\
\]
for each $c\in (0,1)$. Thus
\[
 \liminf_{t\to\infty}\frac{A(t)}{A\big(\frac 1c t\big)}\leq g(c)\leq\limsup_{t\to\infty}\frac{A(t)}{A\big(\frac 1c t\big)},\
\]
and substituting $ct$ in the place of $t$ we get
\[
\liminf_{t\to\infty}\frac{A(ct)}{A(t)}\leq g(c)\leq \limsup_{t\to\infty}\frac{A(ct)}{A(t)}\ ,
\]
which was to show. The proof of implication ii)$\Rightarrow$ i) is analogous, thus we omit it.

\end{proof}

\begin{Remark}
{\rm Let us notice that if a distribution function $g:[0,1]\to [0,1]$ fulfils the conditions i) and ii) in the above theorem, then it need not to be a distribution function of the sequence of blocks $A_n$. 

Let $A=\NN\cap\bigcup_{k=1}^{\infty} ((2k)!,(2k+1)!]$. Then for each $c\in (0,1)$ we have $\liminf_{t\to\infty}\frac{A(ct)}{A(t)}=\lim_{k\to\infty}\frac{A((2k)!)}{A\left(\frac{1}{c}(2k)!\right)}=0$ and $\limsup_{t\to\infty}\frac{A(ct)}{A(t)}=\lim_{k\to\infty}\frac{A((2k+1)!)}{A\left(\frac{1}{c}(2k+1)!\right)}=1$. Hence, the function $g(x)=\min\{2x,1\}$ satisfies $\liminf_{t\to\infty}\frac{A(ct)}{A(t)}\leq g(c)\leq\limsup_{t\to\infty}\frac{A(ct)}{A(t)}$ for $c\in (0,1)$. Assume by contrary that there is a subsequence $(n_j)$ such that $g(x)=\lim_{j\to\infty}F(A_{n_j},x)$ for each $x\in [0,1]$. In particular, $\lim_{j\to\infty}F\left(A_{n_j},\frac{1}{4}\right)=\frac{1}{2}$ and $\lim_{j\to\infty}F\left(A_{n_j},\frac{1}{2}\right)=1$. Let $a_{n_j}\in ((2k)!,(2k+1)!]$, where $k$ is sufficiently large, be such that $F\left(A_{n_j},\frac{1}{4}\right)<\frac{5}{8}$ and $F\left(A_{n_j},\frac{1}{2}\right)>\frac{7}{8}$. Then $\frac{1}{2}a_{n_j}\leq (2k)!$. Otherwise, we would have $a_{n_j}> 2\cdot (2k)!$ and
\begin{align*}
&\frac{7}{8}<F\left(A_{n_j},\frac{1}{2}\right)=\frac{1}{n_j}\#\left\{i\leq n_j: \frac{a_i}{a_{n_j}}<\frac{1}{2}\right\}=\frac{1}{n_j}\#\left\{i\leq n_j: a_i<\frac{1}{2}a_{n_j}\right\}\\
&<\frac{(2k-1)!+\frac{1}{2}a_{n_j}-(2k)!}{(2k-1)!-(2k-2)!+a_{n_j}-(2k)!}=\frac{\frac{1}{2k}+\frac{1}{2}\frac{a_{n_j}}{(2k)!}-1}{\frac{2k-2}{2k(2k-1)}+\frac{a_{n_j}}{(2k)!}-1}<\frac{1}{2},
\end{align*}
which is impossible. Thus $\frac{1}{2}a_{n_j}\leq (2k)!$ and
\begin{align*}
&\frac{7}{8}<F\left(A_{n_j},\frac{1}{2}\right)=\frac{1}{n_j}\#\left\{i\leq n_j: a_i<\frac{1}{2}a_{n_j}\right\}<\frac{(2k-1)!}{(2k-1)!-(2k-2)!+a_{n_j}-(2k)!}.
\end{align*}
Hence, $a_{n_j}<(2k)!+\frac{1}{7}(2k-1)!+(2k-2)!$. On the other hand,
\begin{align*}
&\frac{5}{8}>F\left(A_{n_j},\frac{1}{4}\right)=\frac{1}{n_j}\#\left\{i\leq n_j: a_i<\frac{1}{4}a_{n_j}\right\}>\frac{(2k-1)!-(2k-2)!}{(2k-1)!+a_{n_j}-(2k)!}.
\end{align*}
Hence, $a_{n_j}>(2k)!+\frac{3}{5}(2k-1)!-\frac{8}{5}(2k-2)!$. Finally, we obtain
$$(2k)!+\frac{1}{7}(2k-1)!+(2k-2)!>(2k)!+\frac{3}{5}(2k-1)!-\frac{8}{5}(2k-2)!,$$
which holds only if $2k-1<\frac{91}{16}$, i.e. $k<\frac{107}{32}$. This stays in contradiction with the fact that $k$ can be sufficiently large.
}
\end{Remark}

The above remark shows that not always there exists a distribution function $g(x)$ of $(A_n)$ such that $g(c)=\gamma$ and $g(d)=\delta$ for some $c,d,\gamma,\delta\in [0,1]$ satisfying $\liminf_{t\to\infty}\frac{A(ct)}{A(t)}\leq\gamma\leq\limsup_{t\to\infty}\frac{A(ct)}{A(t)}$ and $\liminf_{t\to\infty}\frac{A(dt)}{A(t)}\leq\delta\leq\limsup_{t\to\infty}\frac{A(dt)}{A(t)}$. Despite this, for any $c,\gamma\in [0,1]$ satisfying $\liminf_{t\to\infty}\frac{A(ct)}{A(t)}\leq\gamma\leq\limsup_{t\to\infty}\frac{A(ct)}{A(t)}$ there exists a distribution function $g(x)$ of $(A_n)$ such that $g(c)=\gamma$. In order to prove this fact, we will need the following lemma.

\begin{Lemma}\label{L2}
Let $c\in (0,1)$ and $\gamma\in\left[\liminf_{t\to\infty}\frac{A(ct)}{A(t)}, \limsup_{t\to\infty}\frac{A(ct)}{A(t)}\right]$. Then there exists an increasing sequence $(n_k)$ of positive integers such that $\lim_{k\to\infty}\frac{A\left(ca_{n_k}\right)}{n_k}=\gamma$.
\end{Lemma}

\begin{proof}
Every value from the interval $\Big[\liminf\limits_{t\to\infty}\frac{A(ct)}{A(t)}, \limsup\limits_{t\to\infty}\frac{A(ct)}{A(t)}\Big]$ is attained as a limit $\lim_{k\to\infty}\frac{A(ct_k)}{A(t_k)}$ for some sequence $(t_k)$ of real numbers increasing to infinity. This holds because the function $[a_1,\infty)\ni t\mapsto\frac{A(ct)}{A(t)}\in [0,\infty)$ is a piecewise constant function with the points of discontinuity of the form $a_n$ and $\frac{1}{c}a_n$, where $n\in\NN$, and variations at these points not greater than $\frac{1}{n}$. Moreover, if $\frac{A(ct)}{A(t)}<\lim_{s\to t^-}\frac{A(cs)}{A(s)}$, then $t=a_n$ for some $n\in\NN$. Hence every value from the interval $\left[\liminf_{t\to\infty}\frac{A(ct)}{A(t)}, \limsup_{t\to\infty}\frac{A(ct)}{A(t)}\right]$ can be attained as $\lim_{k\to\infty}\frac{A\left(ca_{n_k}\right)}{n_k}$ for some increasing sequence $(n_k)$ of positive integers.
\end{proof}

At this moment we are ready to state and prove the aforementioned result.

\begin{Theorem}\label{T2}
Let $c_0\in (0,1)$ and $\gamma\in\left[\liminf_{t\to\infty}\frac{A(c_0t)}{A(t)}, \limsup_{t\to\infty}\frac{A(c_0t)}{A(t)}\right]$. Then there exists $g(x)\in G(A_n)$ such that $g(c_0)=\gamma$.
\end{Theorem}

\begin{proof}
By Lemma \ref{L2} there exists an increasing sequence $(n_k^{(0)})_{k\in\NN}$ of positive integers such that $\lim_{k\to\infty}F\left(A_{n_k^{(0)}},c_0\right)=\lim_{k\to\infty}\frac{A\left(ca_{n_k}\right)}{n_k^{(0)}}=\gamma$. Let $(\QQ\cap (0,1))\bs\{c_0\}=\{c_l: l\in\NN\}$. Since the interval $[0,1]$ is compact (with respect to the natural topology on $\RR$), for each $l\in\NN$ we can choose a subsequence $(n_k^{(l)})_{k\in\NN}$ of the sequence $(n_k^{(l-1)})_{k\in\NN}$ such that the sequence $\left(F\left(A_{n_k^{(l)}},c_l\right)\right)_{k\in\NN}$ is convergent. From the construction of sequences $(n_k^{(l)})_{k\in\NN}$, $l\in\NN$, it follows that the sequence $\left(F\left(A_{n_k^{(l)}},c_j\right)\right)_{k\in\NN}$ is convergent for each $j\in\{0,1,\ldots,l\}$. Thus, the sequence $\left(F\left(A_{n_k^{(k)}},c_j\right)\right)_{k\in\NN}$ is convergent for each $j\in\NN$. Define $g(x)=\limsup_{k\to\infty}F\left(A_{n_k^{(k)}},c_j\right)$ for $x\in [0,1]$. $g$ is a non-decreasing function. In particular, $g$ has at most countably many points of discontinuity. Let $x\in (0,1)$ be a point of continuity of $g$. Then 
$$\lim_{\QQ\ni q\to x^-}\lim_{k\to\infty}F\left(A_{n_k^{(k)}},q\right)=\lim_{\QQ\ni q\to x^-}g(q)=g(x)=\lim_{\QQ\ni q\to x^+}g(q)=\lim_{\QQ\ni q\to x^+}\lim_{k\to\infty}F\left(A_{n_k^{(k)}},q\right).$$
On the other hand, 
$$\lim_{\QQ\ni q\to x^-}\lim_{k\to\infty}F\left(A_{n_k^{(k)}},q\right)\leq\liminf_{k\to\infty}F\left(A_{n_k^{(k)}},x\right)\leq\limsup_{k\to\infty}F\left(A_{n_k^{(k)}},x\right)\leq\lim_{\QQ\ni q\to x^+}\lim_{k\to\infty}F\left(A_{n_k^{(k)}},q\right).$$
Hence, $g(x)=\lim_{k\to\infty}F\left(A_{n_k^{(k)}},x\right)$ for every $x\in [0,1]$ with at most countably many exceptions. This means, that $g(x)\in G(A_n)$.
\end{proof}

Now we give a series of corollaries following from Theorem \ref{T2}.

\begin{Corollary}\label{C0}
If $G(A_n)$ is not a singleton, then $\# G(A_n)=\mathfrak{c}$.
\end{Corollary}

\begin{proof}
Since distribution functions are non-decreasing and we identify two of them if they coincide on their common points of continuity, we can assume that all the elements of $G(A_n)$ are right-continuous functions. Assume that $1<\# G(A_n)<\mathfrak{c}$. Then we consider the set
$$\cal{D}=\{b\in [0,1]: b\text{ is a point of discontinuity of some }g(x)\in G(A_n)\}.$$
Then $\#\cal{D}\leq\# G(A_n)<\mathfrak{c}$. Thus $(0,1)\bs\cal{D}$ is dense in $[0,1]$. We claim that there exist $g_1(x),g_2(x)\in G(A_n)$ and $c_0\in (0,1)\bs\cal{D}$ such that $g_1(c_0)<g_2(c_0)$. Indeed, if any two functions from $G(A_n)$ (that are monotone) coincide on $c\in (0,1)\bs\cal{D}$ (which is dense in $[0,1]$), then they are equal. This means that $\# G(A_n)=1$, contradicting our assumption on cardinality of $G(A_n)$. Let $c_0\in (0,1)\bs\cal{D}$ be such that $g_1(c_0)<g_2(c_0)$. By Theorem \ref{T1} we have 
$$\liminf_{n\to\infty}F(A_n,c)\leq g_1(c)<g_2(c)\leq\limsup_{n\to\infty}F(A_n,c),$$
meanwhile by Theorem \ref{T2} we conclude that for each $\gamma\in\left[\liminf_{n\to\infty}F(A_n,c),\limsup_{n\to\infty}F(A_n,c)\right]$ there exists $g(x)\in G(A_n)$ such that $g(c)=\gamma$. Since $c$ is a point of continuity of any function from $G(A_n)$, we obtain $\mathfrak{c}$ pairwise distinct functions from $G(A_n)$, once again contradicting the assumption on cardinality of $G(A_n)$. Finally, we get that if $G(A_n)$ is not a singleton, then $\# G(A_n)\geq\mathfrak{c}$. On the other hand, the cardinality of $G(A_n)$ does not exceed $\mathfrak{c}$ as the cardinality of the set of all functions mapping the interval $[0,1]$ to itself and having at most countably many points of discontinuity.
\end{proof}

\begin{Remark}
{\rm The fact $\# G(A_n)\in\{1,\mathfrak{c}\}$ was proved implicitly in \cite[Theorem 5.1]{ST} under additional condition on distribution of elements of $A$. To be more precise, the connectedness of $G(A_n)$ with respect to weak topology was showed in this case.

Let us also note that Corollary \ref{C0} follows directly from Theorem \ref{T2} if we distinguish any two functions that differ at at least one point. However, we identify two functions in $G(A_n)$ if they differ on at most countably many points. This is the purpose that the proof of Corollary \ref{C0} becomes more subtle.}
\end{Remark}

\begin{Corollary}\label{C1}
	Let $A=\{a_1<a_2<\cdots\}\subset\NN$. Then, the following conditions are equivalent:
	\begin{itemize}
		\item[(i)]  $G(A_n)=\{g(x)\}$,
		\item[(ii)]   $\forall\ c>0:\ \lim\limits_{t\to\infty}\frac{A(ct)}{A(t)}$ exists,
		\item[(iii)]  $\forall\ c>1:\ \lim\limits_{t\to\infty}\frac{A(ct)}{A(t)}=\frac{1}{g\big(\frac 1c\big)}$,
		\item[(iv)]  $\forall\  0<c<1:\ \lim\limits_{t\to\infty}\frac{A(ct)}{A(t)}=g(c)$.
	\end{itemize}	
\end{Corollary}
\begin{proof}
According to Theorem~\ref{T1}, it suffices to prove the equvalence i) $\Leftrightarrow$ ii). For the proof of ii) $\Rightarrow$ i) assume that $\lim\limits_{t\to\infty}\frac{A(ct)}{A(t)}$ exists for each $c>0$. By Theorem \ref{T1}, each distribution fuction for $(A_n)$ is of the form $g(x)=\lim\limits_{t\to\infty}\frac{A(xt)}{A(t)}$, $x\in (0,1)$.

We prove the implication i)$\Rightarrow$ii) by contrary. If ii) does not hold, i.e. $\liminf\limits_{t\to\infty}\frac{A(ct)}{A(t)}\neq\limsup\limits_{t\to\infty}\frac{A(ct)}{A(t)}$ for some $c\in (0,1)$, then the set $G(A_n)$ contains functions $g_1(x)$ and $g_2(x)$ such that $g_1(c)=\liminf\limits_{t\to\infty}\frac{A(ct)}{A(t)}$ and $g_2(c)=\limsup\limits_{t\to\infty}\frac{A(ct)}{A(t)}$. However, we assume i), i.e. $G(A_n)=\{g(x)\}$. Thus, $g(x)\in\{c_0(x),x^q: q\in (0,1]\}$ in virtue of (A1). Since $g_1$, $g_2$ and $g$ are non-decreasing and $g|_{(0,1)}$ is continuous, then there exists $i\in\{1,2\}$ and $\varepsilon>0$ such that $g_i(x)\neq g(x)$ for each $x\in(c-\varepsilon,c+\varepsilon)$. Hence $G(A_n)$ has at least two elements, contradicting i). 
\end{proof}

\begin{Corollary}\label{C2}
Let $A=\{a_1<a_2<\cdots\}\subset\NN$ and $0<q\leq1$. Then, we have
\begin{itemize}
  \item[(i)]  $A\in\Uc\ \Longleftrightarrow\ \forall\  c>0:\ \lim_{t\to\infty}\frac{A(ct)}{A(t)}=1$, i.e. $A(x)\in \RV_0$.
  \item[(ii)]   $A\in\Uq\ \Longleftrightarrow\ \forall\  c>0:\ \lim_{t\to\infty}\frac{A(ct)}{A(t)}=c^q$ , i.e. $A(x)\in \RV_q$.
\end{itemize}	
\end{Corollary}	

\begin{proof}
	This is a direct corollary of Corollary~\ref{C1} and (A1).
\end{proof}

\begin{Corollary}\label{C3}
Let $A=\{a_1<a_2<\cdots\}\subset\NN$. Then, the following conditions are equivalent:
\begin{itemize}
  \item[(i)]  $\forall\ c>1:\ \lim_{t\to\infty}\frac{A(ct)}{A(t)}$ exists,
  \item[(ii)]   $\exists q\in [0,1]: A(x)\in \RV_q$.
\end{itemize}	
\end{Corollary}

\begin{proof}
For the proof of the implication i)$\Rightarrow$ ii) denote $g\left(\frac{1}{c}\right)=\lim_{t\to\infty}\frac{A(t)}{A(ct)}$ for each $c\geq 1$. By Corollary~\ref{C1} we have $G(A_n)=\{g(x)\}$, which means that $A\in\mathcal{U}(g(x))$. Hence, by (A1) we have $g(x)\in\{c_0(x),x^q: q\in (0,1]\}$ and by Corollary \ref{C2} we conclude that $A(x)\in \RV_q$ for some $q\in [0,1]$.

The implication ii)$\Rightarrow$ i) is obvious.
\end{proof}

From Corollary~\ref{C2}, (A7) and (A8) we obtain the following.

\begin{Corollary}\label{C4}
Let $A=\{a_1<a_2<\cdots\}\subset\NN$ and $0<q\leq1$. Then we have	
\begin{itemize}
	\item [(i)] If $\lim_{t\to\infty}\frac{A(ct)}{A(t)}=1$,\ for every $c>0$ then $\lim_{n\to\infty}\frac{\log n}{\log a_n}=0,\ \textrm{i.e.}\ \lambda(A)=0$.
	\item [(ii)] If $\lim_{t\to\infty}\frac{A(ct)}{A(t)}=c^q$,\ for every $c>0$ then $\lim_{n\to\infty}\frac{\log n}{\log a_n}=q\ (\textrm{and then}\ \lambda(A)=q)$.
\end{itemize}
\end{Corollary}

\begin{Remark}\label{liminfnotimpRVq}
{\rm Of course, the condition $\lim_{t\to\infty}\frac{A(ct)}{A(t)}=1$ for each $c>1$ implies that $\lambda(A)=0$. The inverse implication does not hold as $\Uc\subsetneq\Ic$ \cite[Example 1]{TBFM}. This means that there exist sets $B\subset A\subset\NN$ such that $A\in\Uc$ but $B\not\in\Uc$. By Corollary \ref{C2} we thus have that $B(x)\not\in\RV_0$. However, since $A(x)\in\RV_0$, by Corollary \ref{C3} we conclude that $\lim_{n\to\infty}\frac{\log A(n)}{\log n}=\lambda(A)=0$. Hence, $\lim_{n\to\infty}\frac{\log B(n)}{\log n}=\lambda(B)=0$ as $B\subset A$ and the exponent of convergence is a monotone set function.}
\end{Remark}

\begin{Remark}\label{liminfnotimplim}
{\rm Let us note that the condition $\liminf_{t\to\infty}\frac{A(ct)}{A(t)}>1$ for each $c>1$ does not imply that $\lim_{t\to\infty}\frac{A(ct)}{A(t)}$ exists. Let
$$A=\NN\cap\bigcup_{k=1}^{\infty}\left(((2k-1)!,(2k)!]\cup\{2j:\ (2k)!<2j\leq (2k+1)!\}\right).$$
Fix arbitrary $c>1$. Then
\begin{align*}
&\liminf_{t\to\infty}\frac{A(ct)}{A(t)}=\lim_{k\to\infty}\frac{A(c\cdot (2k)!)}{A((2k)!)}=1+\lim_{k\to\infty}\frac{A(c\cdot (2k)!)-A((2k)!)}{A((2k)!)}=1+\lim_{k\to\infty}\frac{\frac{(c-1)\cdot (2k)!}{2}}{(2k)!}\\
&=1+\frac{c-1}{2}=\frac{c+1}{2}>1
\end{align*}
and
\begin{align*}
&\limsup_{t\to\infty}\frac{A(ct)}{A(t)}=\lim_{k\to\infty}\frac{A(c\cdot (2k-1)!)}{A((2k-1)!)}=1+\lim_{k\to\infty}\frac{A(c\cdot (2k-1)!)-A((2k-1)!)}{A((2k-1)!)}\\
&=1+\lim_{k\to\infty}\frac{(c-1)\cdot (2k-1)!}{\frac{(2k-1)!}{2}}=1+2(c-1)=2c-1>\frac{c+1}{2}.
\end{align*}}
\end{Remark}

The example below shows that for each $q\in (0,1]$ there exists a set $A\subset\NN$ such that the limit $\lim_{n\to\infty}\frac{\log A(n)}{\log n}$ does not exist and $\lambda(A)=q$.

\begin{Example}\label{nolim}
{\rm Let $0\leq p<q\leq 1$ be fixed. Let $(p_k)$ be a non-increasing sequence of positive numbers convregent to $p$ and $(b_k)$ be a strictly increasing sequence of positive integers such that $$\frac{b_1^q}{b_2^p}<1,\ \frac{b_k^q}{b_{k+1}^{p_{\lceil k/2\rceil}}}\downarrow 0\ \textrm{as}\ k\to\infty\ \textrm{and}\ \lim_{k\to\infty}\frac{\log k}{\log b_k}=0\ .$$ 
For $k\in\NN$ we define sets
\begin{align*}
A_{2k-1}=&\left\{\left\lceil j^{\frac{1}{p_k}}\right\rceil:j\in\NN\ \mbox{and}\ b_{2k-1}<j^{\frac{1}{p_k}}\leq b_{2k}\right\},\\
A_{2k}=&\left\{\left\lceil j^{\frac{1}{q}}\right\rceil:j\in\NN\ \mbox{and}\ b_{2k}<j^{\frac{1}{q}}\leq b_{2k+1}\right\}.
\end{align*}
Choose the set
$$A=\bigcup_{k=1}^{\infty}A_k\ .$$
Then,
\begin{align*}
&\lim_{k\to\infty}\frac{\log A(b_{2k+1})}{\log b_{2k+1}}\geq\lim_{k\to\infty}\frac{\log\#A_{2k}}{\log b_{2k+1}}\geq\lim_{k\to\infty}\frac{\log (b_{2k+1}^q-b_{2k}^q-1)}{\log b_{2k+1}}\\
&=\lim_{k\to\infty}\frac{\log b_{2k+1}^q+\log \big(1-\frac{b_{2k}^q}{b_{2k+1}^q}-\frac1{b_{2k+1}^q}\big)}{\log b_{2k+1}}=\lim_{k\to\infty}\frac{\log b_{2k+1}^q}{\log b_{2k+1}}=q
\end{align*}
(because the second equality in the second line holds due to the fact that $\lim_{k\to\infty}\frac{b_{2k}^q}{b_{2k+1}^q}=0$) and
\begin{align*}
&\lim_{k\to\infty}\frac{\log A(b_{2k})}{\log b_{2k}}=\lim_{k\to\infty}\frac{\log(\#A_{2k-1}+\sum_{j=1}^{k-1}(\# A_{2j-1}+\# A_{2j}))}{\log b_{2k}}\\
&\leq\lim_{k\to\infty}\frac{\log (b_{2k}^{p_k}-b_{2k-1}^{p_k}+\sum_{j=1}^{k-1}(b_{2j}^{p_j}-b_{2j-1}^{p_j}+b_{2j+1}^q-b_{2j}^q))}{\log b_{2k}}\leq\lim_{k\to\infty}\frac{\log (2k-1)(b_{2k}^{p_k}-b_{2k-1}^{p_k})}{\log b_{2k}}\\
&=\lim_{k\to\infty}\frac{\log (2k-1)+\log b_{2k}^{p_k}+\log \Big(1-\frac{b_{2k-1}^{p_k}}{b_{2k}^{p_k}}\Big)}{\log b_{2k}}=\lim_{k\to\infty}\frac{\log b_{2k}^{p_k}}{\log b_{2k}}=p
\end{align*}
(because the second inequality in the second line follows from the fact that $b_{2j}^{p_j}-b_{2j-1}^{p_j}\leq b_{2j}^q-b_{2j-1}^q\leq b_{2k}^{p_k}-b_{2k-1}^{p_k}$ and $b_{2j+1}^q-b_{2j}^q\leq b_{2k}^{p_k}-b_{2k-1}^{p_k}$ for every $j=1, 2,\dots, k-1$, since $\frac{b_{2k-1}^q}{b_{2k}^{p_k}}\downarrow 0$ as $k\to\infty$, and the second equality in the third line comes from the facts that $\lim_{k\to\infty}\frac{b_{2k-1}^{p_k}}{b_{2k}^{p_k}}=0$ and $\lim_{k\to\infty}\frac{\log (2k-1)}{\log b_{2k}}=0$). Hence,
$$\liminf_{n\to\infty}\frac{\log A(n)}{\log n}\leq p<q\leq\limsup_{n\to\infty}\frac{\log A(n)}{\log n}\ .$$
On the other hand,
\[
\limsup_{n\to\infty}\frac{\log A(n)}{\log n}\leq\lim_{n\to\infty}\frac{\log\#\{j\in\NN: j^{\frac{1}{q}}\leq n\}}{\log n}\leq\lim_{n\to\infty}\frac{\log n^q}{\log n}=q
\]
and
\begin{align*}
&\liminf_{n\to\infty}\frac{\log A(n)}{\log n}\geq\lim_{n\to\infty}\frac{\log\#\{j\in\NN: j^{\frac{1}{p}}\leq n\}}{\log n}\geq\lim_{n\to\infty}\frac{\log n^p}{\log n}=p\ .
\end{align*}
for $p>0$ (if $p=0$, then obviously $\liminf_{n\to\infty}\frac{\log A(n)}{\log n}\geq p$). Hence,
$$p\leq\liminf_{n\to\infty}\frac{\log A(n)}{\log n}\leq\limsup_{n\to\infty}\frac{\log A(n)}{\log n}\leq q\ .$$
From (A15) we have as a result, $\lambda(A)=\limsup_{n\to\infty}\frac{\log A(n)}{\log n}=q$ and $\liminf_{n\to\infty}\frac{\log A(n)}{\log n}=p$.}
\end{Example} 

The following lemma shows that there is a strong connection between topological properties of directions sets and multidimensional ratio sets.

\begin{Lemma}\label{directionratio}
Let $k\in\NN$ and $A_1,\ldots,A_{k-1},B\subset\NN$. Then, the following conditions are equivalent:
\begin{itemize}
\item[(i)] $D^k(A_1,\ldots,A_{k-1};B)$ is dense in $\SS_+^{k-1}$;
\item[(ii)] $R^k(A_1,\ldots,A_{k-1};B)$ is dense in $\RR_+^{k-1}$.
\end{itemize}
If additionally $A_1=\cdots=A_{k-1}=B=A$, then the above conditions are equivalent to the following one:
\begin{itemize}
\item[(iii)] $R^k(A)$ is dense in $(0,1)^{k-1}$.
\end{itemize}
\end{Lemma}

\begin{proof}
We start with the proof of the equivalence i)$\Leftrightarrow$ ii). Consider the mappings
\begin{align*}
F:\SS_+^{k-1}\ni (x_1,\ldots,x_k)\mapsto &\left(\frac{x_1}{x_k},\ldots,\frac{x_{k-1}}{x_k}\right)\in\RR_+^{k-1},\\
G:\RR_+^{k-1}\ni (y_1,\ldots,y_{k-1})\mapsto &\left(\frac{y_1}{\sqrt{y_1^2+\ldots+y_k^2+1}},\ldots,\frac{y_{k-1}}{\sqrt{y_1^2+\ldots+y_k^2+1}},\frac{1}{\sqrt{y_1^2+\ldots+y_k^2+1}}\right)\in\SS_+^{k-1}.
\end{align*}
These mappings are continuous and inverse to each other. Indeed, let us compute \linebreak$(G\circ F)(x_1,\ldots, x_k)$:
\begin{align*}
&(G\circ F)(x_1,\ldots, x_k)=G\left(\frac{x_1}{x_k},\ldots,\frac{x_{k-1}}{x_k}\right)\\
&=\left(\frac{\frac{x_1}{x_k}}{\sqrt{\frac{x_1^2}{x_k^2}+\ldots+\frac{x_{k-1}^2}{x_k^2}+1}},\ldots,\frac{\frac{x_{k-1}}{x_k}}{\sqrt{\frac{x_1^2}{x_k^2}+\ldots+\frac{x_{k-1}^2}{x_k^2}+1}},\frac{1}{\sqrt{\frac{x_1^2}{x_k^2}+\ldots+\frac{x_{k-1}^2}{x_k^2}+1}}\right)\\
&=\left(\frac{x_1}{x_k\ \sqrt{\frac{x_1^2}{x_k^2}+\ldots+\frac{x_{k-1}^2}{x_k^2}+1}},\ldots,\frac{x_{k-1}}{x_k\ \sqrt{\frac{x_1^2}{x_k^2}+\ldots+\frac{x_{k-1}^2}{x_k^2}+1}},\frac{x_k}{x_k\ \sqrt{\frac{x_1^2}{x_k^2}+\ldots+\frac{x_{k-1}^2}{x_k^2}+1}}\right)\\
&=\left(\frac{x_1}{\sqrt{x_1^2+\ldots+x_{k-1}^2+x_k^2}},\ldots,\frac{x_{k-1}}{\sqrt{x_1^2+\ldots+x_{k-1}^2+x_k^2}},\frac{x_k}{\sqrt{x_1^2+\ldots+x_{k-1}^2+x_k^2}}\right)\\
&=(x_1,\ldots,x_{k-1},x_k),
\end{align*}
where in the last equality we use the fact, that $(x_1,\ldots, x_k)\in\SS^{k-1}$. Now, we compute \linebreak$(F\circ G)(y_1,\ldots,y_{k-1})$:
\begin{align*}
&(F\circ G)(y_1,\ldots,y_{k-1})\\
&=F\left(\frac{y_1}{\sqrt{y_1^2+\ldots+y_{k-1}^2+1}},\ldots,\frac{y_{k-1}}{\sqrt{y_1^2+\ldots+y_{k-1}^2+1}},\frac{1}{\sqrt{y_1^2+\ldots+y_{k-1}^2+1}}\right)\\
&=(y_1,\ldots,y_{k-1}).
\end{align*}
Let $a_1\in A_1,\ldots, a_{k-1}\in A_{k-1}, b\in B$. Then,
\begin{align*}
&F\left(\frac{a_1}{\sqrt{a_1^2+\ldots+a_{k-1}^2+b^2}},\ldots,\frac{a_{k-1}}{\sqrt{a_1^2+\ldots+a_{k-1}^2+b^2}},\frac{b}{\sqrt{a_1^2+\ldots+a_{k-1}^2+b^2}}\right)\\
&=\left(\frac{a_1}{b},\ldots,\frac{a_{k-1}}{b}\right).
\end{align*}
Hence, $F:\SS_+^{k-1}\to\RR_+^{k-1}$ is a homeomorphism such that $$F(D^k(A_1,\ldots,A_{k-1};B))=R^k(A_1,\ldots,A_{k-1};B).$$ Thus, the denseness of $D^k(A_1,\ldots,A_{k-1};B)$ in $\SS_+^{k-1}$ is equivalent to the denseness of \linebreak$R^k(A_1,\ldots,A_{k-1};B)$ in $\RR_+^{k-1}$.

From now on we assume that $A_1=\ldots=A_{k-1}=B=A$. Then, the implication ii)$\Rightarrow$ iii) is obvious. It remains to prove iii)$\Rightarrow$ ii). Let us consider the sets $X_j=\{(x_1,\ldots,x_{k-1})\in\RR_+^{k-1}: x_j>1, \forall_{i\neq j}\ x_i<x_j\}$, where $j\in\{1,\ldots,k-1\}$. Then, the mapping
\begin{align*}
H_j:\RR_+^{k-1}\ni (x_1,\ldots,x_{k-1})\mapsto (x_1x_j^{-1},\ldots,x_{j-1}x_j^{-1},x_j^{-1},x_{j+1}x_j^{-1},\ldots,x_{k-1}x_j^{-1})\in\RR_+^{k-1}
\end{align*}
is continuous and involutive. Moreover, $H_j((0,1)^{k-1})=X_j$ and $H_j(R^k(A))=R^k(A)$, as
\begin{align*}
H_j\left(\frac{a_1}{a_k}\ldots,\frac{a_{k-1}}{a_k}\right)=\left(\frac{a_1}{a_j}\ldots,\frac{a_{j-1}}{a_j},\frac{a_k}{a_j},\frac{a_{j+1}}{a_j},\ldots,\frac{a_{k-1}}{a_j}\right).
\end{align*}
Consequently, if $R^k(A)$ is dense in $(0,1)^{k-1}$, then $R^k(A)$ is dense in $X_j$ for each $j\in\{1,\ldots,k-1\}$. This, combined with the facts that $(0,1)^{k-1}$ is dense in $(0,1]^{k-1}$, $X_j$ is dense in the set $Y_j=\{(x_1,\ldots,x_{k-1})\in\RR_+^{k-1}: x_j\geq 1, \forall_{i\neq j}\ x_i\leq x_j\}$, $j\in\{1,\ldots,k-1\}$, and $(0,1]^{k-1}\cup\bigcup_{j=1}^{k-1}Y_j=\RR_+^{k-1}$, gives the denseness of $R^k(A)$ in $\RR_+^{k-1}$.
\end{proof}

Now, we give a list of equivalent conditions to the $(N)$-denseness of a given subset of $\NN$.

\begin{Theorem}\label{Ndense}
Let $A=\{a_1<a_2<\cdots\}\subset\NN$. Then, the following conditions are equivalent:
\begin{itemize}
\item[i)] $\forall\  c>1:\ A(ct)>A(t)\ \textrm{for}\ t\gg 0$;
\item[ii)] $\forall\  c>1:\ \lim_{t\to\infty}(A(ct)-A(t))=\infty$;
\item[iii)] $\forall\ k\geq 2\ \forall\ B\subset\NN\ \textrm{infinite:}\ R^k(A;B)\ \textrm{is dense in}\ \RR_+^{k-1}$;
\item[iv)] $\forall\ k\geq 2\ \forall\ B\subset\NN\ \textrm{infinite:}\ D^k(A;B)\ \textrm{is dense in}\ \SS_+^{k-1}$;
\item[v)] $\exists\ k\geq 2\ \forall\ B\subset\NN\ \textrm{infinite:}\ R^k(A;B)\ \textrm{is dense in}\ \RR_+^{k-1}$;
\item[vi)] $\exists\ k\geq 2\ \forall\ B\subset\NN\ \textrm{infinite:}\ D^k(A;B)\ \textrm{is dense in}\ \SS_+^{k-1}$;
\item[vii)] $\forall\ B\subset\NN\ \textrm{infinite}\ \exists\ k\geq 2:\ R^k(A;B)\ \textrm{is dense in}\ \RR_+^{k-1}$;
\item[viii)] $\forall\ B\subset\NN\ \textrm{infinite}\ \exists\ k\geq 2:\ D^k(A;B)\ \textrm{is dense in}\ \SS_+^{k-1}$;
\item[ix)] $A\ \textrm{is}\ (N)-\textrm{dense}$;
\item[x)] $\lim_{n \to \infty}\frac{a_{n+1}}{a_n}=1$;
\item[xi)] $\lim_{n \to \infty}\frac{1}{a_n}\max\{a_1,a_{i+1}-a_i: i\in\{1,\ldots,n-1\}\}=0$.
\end{itemize}
\end{Theorem}

\begin{proof}
In order to show the implication i)$\Rightarrow$ ii) it suffices to show that for each $c>1$ and $k\in\NN$ there exists $t_{c,k}$ such that $A(ct)-A(t)\geq k$ for each $t\geq t_{c,k}$. Indeed, let $t_{c,k}$ be such that $A(\sqrt[k]{c}t)>A(t)$ for $t\geq t_{c,k}$. Then
$$A(ct)-A(t)=\sum_{j=1}^k \left(A\left(\sqrt[k]{c}^jt\right)-A\left(\sqrt[k]{c}^{j-1}t\right)\right)\geq k,$$
which was to prove.

The implication ii)$\Rightarrow$ i) is trivial.

For the proof of the implication i)$\Rightarrow$ iii) let us fix $k\geq 2$, an infinite subset $B$ of $\NN$ and $\alpha_1,\beta_1\ldots,\alpha_{k-1},\beta_{k-1}\in\RR_+$ such that $\alpha_j<\beta_j$ for each $j\in\{1,\ldots,k-1\}$. By i), for sufficiently large $t>0$ we have $A(\beta_jt)>A(\alpha_jt)$ for each $j\in\{1,\ldots,k-1\}$. Since $B\subset\NN$ is infinite, we can choose $b\in B$ so large that $A(\beta_jb)>A(\alpha_jb)$ for each $j\in\{1,\ldots,k-1\}$. This means that  for each $j\in\{1,\ldots,k-1\}$ there exists an $a_j\in A$ such that $\alpha_jb<a_j\leq\beta_jb$, equivalently $\alpha_j<\frac{a_j}{b}\leq\beta_j$. Hence, $\left(\frac{a_1}{b},\ldots,\frac{a_{k-1}}{b}\right)\in R^k(A;B)\cap (\alpha_1,\beta_1]\times\ldots\times (\alpha_{k-1},\beta_{k-1}]$. Because $\alpha_1,\beta_1\ldots,\alpha_{k-1},\beta_{k-1}$ are chosen arbitrarily, we conclude that the set $R^k(A;B)$ is dense in $\RR_+^{k-1}$.

The equivalences iii)$\Leftrightarrow$ iv), v)$\Leftrightarrow$ vi) and vii)$\Leftrightarrow$ viii) follow directly from Lemma \ref{directionratio}.

The implications iii)$\Rightarrow$ v) and v)$\Rightarrow$ vii) are obvious.

The implication vii)$\Rightarrow$ ix) follows from the facts that $R(A,B)$ is an image of $R^k(A;B)$ via the projection $\RR_+^{k-1}\ni (x_1,\ldots,x_{k-1})\mapsto x_1\in\RR_+$ and an image of a dense subset via continuous surjection is a dense subset of a codomain.

We prove the implication ix)$\Rightarrow$ i) by contraposition. Assume that there exist a $c_0>1$ and a strictly increasing sequence $(x_n)_{n\in\NN}$ of positive numbers such that $A(c_0x_n)=A(x_n)$ for each $n\in\NN$. Then, take $N\in\NN$ such that $\frac{x_N}{x_N-1}<c_0$ and put $c_1=\frac{x_N}{x_N-1}$. Thus, for each $n\geq N$ we have $\frac{x_n}{x_n-1}\leq c_1$. This means that
$$x_n\leq c_1(x_n-1)<c_1\lfloor x_n\rfloor<c_0\lfloor x_n\rfloor\leq c_0x_n, \quad n\geq N.$$
Hence, $A(c_1\lfloor x_n\rfloor)=A(c_0\lfloor x_n\rfloor)$ for $n\geq N$. Let us consider the set $B=\{\lfloor x_n\rfloor : n\geq N\}$. If $R(A,B)\cap (c_1,c_0)\neq\varnothing$, then there are some $a\in A$ and $n\geq N$ such that $c_1<\frac{a}{\lfloor x_n\rfloor}<c_0$, or equivalently $c_1\lfloor x_n\rfloor<a<c_0\lfloor x_n\rfloor$ - this contradicts with the fact that $A(c_1\lfloor x_n\rfloor)=A(c_0\lfloor x_n\rfloor)$ for $n\geq N$. Hence, $R(A,B)$ is not dense in $\RR_+$ and, as a result, $A$ is not (N)-dense.

The implication x)$\Rightarrow$ ix) is exactly the fact (A21).

For the proof of the implication i)$\Rightarrow$ x) it suffices to show that for each $\varepsilon>0$ and sufficiently large $n\in\NN$ we have $\frac{a_{n+1}}{a_n}\leq 1+\varepsilon$. Let us fix $\varepsilon>0$. By i), there exists a constant $M>0$ such that $A((1+\varepsilon)t)>A(t)$ holds for every $t\geq M$. As a consequence, if $a_n\geq M$, then $A((1+\varepsilon)a_n)>A(a_n)$, which means that $a_{n+1}\leq (1+\varepsilon)a_n$, equivalently $\frac{a_{n+1}}{a_n}\leq 1+\varepsilon$. This was to prove.

We show the implication x)$\Rightarrow$ xi). For the convenience of notation we put $a_0=0$. Then, for each $n\in\NN$ there exists $i(n)\in\{0,1,\ldots,n-1\}$ such that $\max\{a_{i+1}-a_i: i\in\{0,1,\ldots,n-1\}\}=a_{i(n)+1}-a_{i(n)}$. If the sequence $(i(n))_{n\in\NN}$ is bounded, then the value of $a_{i(n)+1}-a_{i(n)}$ is bounded and
$$\lim_{n\to\infty}\frac{1}{a_n}\max\{a_1,a_{i+1}-a_i: i\in\{1,\ldots,n-1\}\}=\lim_{n\to\infty}\frac{a_{i(n)+1}-a_{i(n)}}{a_n}=0$$
as $\lim_{n\to\infty}a_n=\infty$. If the sequence $(i(n))_{n\in\NN}$ is unbounded, then we write the following chain of inequalities.
\begin{align*}
0<\frac{1}{a_n}\max\{a_{i+1}-a_i: i\in\{0,\ldots,n-1\}\}=\frac{a_{i(n)+1}-a_{i(n)}}{a_n}=\frac{a_{i(n)+1}-a_{i(n)}}{a_{i(n)}}\frac{a_{i(n)}}{a_n}<\frac{a_{i(n)+1}}{a_{i(n)}}-1,
\end{align*}
where $a_0=0$. By x), the right-most expression in the above chain tends to $0$ as $n\to\infty$. Therefore, $\frac{1}{a_n}\max\{a_1,a_{i+1}-a_i: i\in\{1,\ldots,n-1\}\}\to 0$ as $n\to\infty$.

We are left with the proof of the implication xi)$\Rightarrow$ x). We write
\begin{align*}
0<\frac{a_n-a_{n-1}}{a_n}\leq\frac{1}{a_n}\max\{a_1,a_{i+1}-a_i: i\in\{1,\ldots,n-1\}\}.
\end{align*}
By x), the right-most expression in the above double inequality tends to $0$ as $n\to\infty$. Thus, $\frac{a_n-a_{n-1}}{a_n}=1-\frac{a_{n-1}}{a_n}\to 0$ as $n\to\infty$. Hence, $\lim_{n\to\infty}\frac{a_{n-1}}{a_n}=1$, which is equivalent to the limit $\lim_{n\to\infty}\frac{a_{n+1}}{a_n}=1$.
\end{proof}

\begin{Remark}
{\rm The reasoning like in the proof of the implication vii)$\imp$ ix), in view of Lemma \ref{directionratio}, gives another proof of the fact that if $k<l$, then the denseness of $D^l(A)$ in $\SS_+^{l-1}$ implies the denseness of $D^k(A)$ in $\SS_+^{k-1}$.}
\end{Remark}

At this moment we give a generalization of the last implication in (A22).

\begin{Theorem}\label{1/k}
Let $k\in\NN$, $k\geq 2$, and $A\subset\NN$. Assume that the set $R^k(A)$ is dense in $\RR_+^{k-1}$. Then, $\underline{D}(A)\leq\frac{1}{k}$.
\end{Theorem}

\begin{proof}
Let us write $A=\{a_1<a_2<\cdots\}$. Since $R^k(A)$ is dense in $\RR_+^{k-1}$, there exists a sequence $((a_1^{(n)},\ldots,a_k^{(n)}))_{n\in\NN}$ such that $a_j^{(n)}\in A$ for each $j\in\{1,\ldots,k\}$ and $n\in\NN$, and $\lim_{n\to\infty}\frac{a_j^{(n)}}{a_k^{(n)}}=\frac{j}{k}$ for every $j\in\{1,\ldots,k-1\}$. Then,
$$\max_{t\in\NN:a_t\leq a_k^{(n)}}\frac{a_t-a_{t-1}}{a_k^{(n)}}\leq \max_{1\leq j\leq k}\frac{a_j^{(n)}-a_{j-1}^{(n)}}{a_k^{(n)}}\to\frac{1}{k},\quad n\to\infty,$$
where we put $a_0=0$ and $a_0^{(n)}=0$, $n\in\NN$. As a result, $$\underline{D}(A)=\liminf_{s\to\infty}\max_{1\leq t\leq s}\frac{a_t-a_{t-1}}{a_s}\leq\frac{1}{k}.$$
\end{proof}

With the use of the above theorem and Lemma \ref{directionratio} we are able to give a topological characterization of having dispersion equal to zero.

\begin{Theorem}\label{D0}
Let $A\subset\NN$. Then, $\underline{D}(A)=0$ if and only if for each positive integer $k\geq 2$ the set $R^k(A)$ is dense in $\RR_+^{k-1}$.
\end{Theorem}

\begin{proof}
Firstly, we assume that $\underline{D}(A)=0$. For the proof of the implication $\Rightarrow$, by Lemma \ref{directionratio} it suffices to prove that for each positive integer $k\geq 2$ the set $R^k(A)$ is dense in $(0,1)^{k-1}$. Let $A=\{a_1<a_2<\cdots\}$ and put $a_0=0$. Choose any $x_1,\ldots,x_{k-1}\in (0,1)$ and $\varepsilon>0$. Since $\underline{D}(A)=0$, we can choose an $n\in\NN$ such that $\max_{1\leq t\leq n}\frac{a_t-a_{t-1}}{a_n}<\varepsilon$. Then, there exist $t_1,\ldots,t_{k-1}$ such that $\left|\frac{a_{t_j}}{a_n}-x_j\right|<\varepsilon$ for each $j\in\{1,\ldots,k-1\}$, which ends the proof.

The proof of the implication $\Leftarrow$ follows directly from Theorem \ref{1/k}.
\end{proof}

\begin{Remark}
{\rm The equivalence x)$\Leftrightarrow$ xi) in Theorem \ref{Ndense} combined with the implication $\Rightarrow$ in Theorem \ref{D0} gives another proof of (A24).}
\end{Remark}

\begin{Remark}\label{D0notimpNdense}
{\rm Let us notice that the implication in (A24) cannot be inversed. It suffices to construct a set $A=\{a_1<a_2<\cdots\}\subset\NN$ such that $\underline{D}(A)=0$ but the $\limsup_{n\to\infty}\frac{a_{n+1}}{a_n}>1$. Then, by Theorem \ref{D0} the set $D^k(A)$ is dense in $\SS_+^{k-1}$ for each $k\geq 2$.

Let us consider the set $A=\NN\cap\bigcup_{k=1}^{\infty}[(2k)!,(2k+1)!]$. Then,
$$\underline{D}(A)=\lim_{k\to\infty}\frac{(2k)!-(2k-1)!)}{(2k+1)!}=\lim_{k\to\infty}\frac{1}{2k+1}-\frac{1}{2k(2k+1)}=0.$$
On the other hand,
$$\limsup_{n\to\infty}\frac{a_{n+1}}{a_n}=\lim_{k\to\infty}\frac{(2k)!}{(2k-1)!}=\lim_{k\to\infty}2k=\infty.$$}
\end{Remark}

Our next result shows that Theorem \ref{1/k} is the only dependence between the denseness of generalized ratio sets of a given subset of $\NN$ and dispersion of this subset. Moreover, we can construct a subset $A$ of positive integers with $R^k(A)$ dense in $\RR_+^{k-1}$, $R^{k+1}(A)$ not dense in $\RR_+^k$ and arbitrarily prescribed dispersion $d\in\left(0,\frac{1}{k}\right]$ and the value of $\lim_{t\to\infty}\frac{\log A(t)}{\log t}\in [0,1]$.

\begin{Theorem}\label{arbitrarydispersion}
For each positive integer $k\in\NN$ and real numbers $d\in\left(0,\frac{1}{k}\right]$ and $\lambda\in [0,1]$ there exists a set $A\subset\NN$ such that $R^k(A)$ is dense in $\RR_+^{k-1}$, $R^{k+1}(A)$ is not dense in $\RR_+^k$, $\underline{D}(A)=d$ and $\lim_{t\to\infty}\frac{\log A(t)}{\log t}=\lambda$ (in case of $k=1$ we understand the condition ''$R^k(A)$ is dense in $\RR_+^{k-1}$'' as an empty fulfilled one).
\end{Theorem}

\begin{proof}
Let us fix $k\in\NN$, $k\geq 2$, $d\in\left(0,\frac{1}{k}\right]$ and $\lambda\in [0,1]$. Enumerate all the elements of the set $\{(x_1,\ldots,x_{k-1})\in (0,1)\cap\QQ^{k-1}: x_1<x_2<\ldots<x_{k-1}\}=\{(q_1^{(n)},\ldots,q_{k-1}^{(n)}): n\in\NN\}$ such that for each $l\in\{1,\ldots ,k-1\}$ the value of $(4n-1)!\cdot q_l^{(n)}$ is an integer at least equal to $(4n-2)!$. Construct three families of subsets of $\NN$:
\begin{align*}
B_n=&\ \left\{\left\lceil j^{\frac{1}{\lambda}}\right\rceil :\ j\in\NN\right\}\cap \left(\max\{(4n-4)!,d\cdot (4n-3)!\},\frac{2d}{2-d}\cdot (4n-3)!\right],\\
C_n=&\ \{(4n-1)!\cdot q_1^{(n)},\ldots ,(4n-1)!\cdot q_{k-1}^{(n)}, (4n-1)!\},\\
D_n=&\ \left\{\left\lceil \frac{(4n-1)!}{(1-d)^j}\right\rceil :\ j\in\NN\right\}\cap ((4n-1)!,(4n)!],
\end{align*}
where $n\in\NN$ (if $d=1$, then we take $D_n=\varnothing$; if $\lambda=0$, then we put $B_n=\varnothing$). Then we define
$$A=\bigcup_{n\in\NN}(B_n\cup C_n\cup D_n).$$

From the very definition of $C_n$ we know that $(q_1^{(n)},\ldots,q_{k-1}^{(n)})\in R^k(A)$. Additionally, each $(k-1)$-tuple of rational numbers from the interval $(0,1)$ can be attained as $(q_{\sigma(1)}^{(n)},\ldots,q_{\sigma(k-1)}^{(n)})$ for some $\sigma:\{1,\ldots,k-1\}\to\{1,\ldots,k-1\}$ (not necessarily bijective). Thus $R^k(A)$ is dense in $(0,1)^{k-1}$. By Lemma \ref{directionratio}, $R^k(A)$ is dense in $\RR_+^{k-1}$.

Assume by contrary that $R^{k+1}(A)$ is dense in $\RR_+^k$. Let $(r_1,\ldots,r_k)\in (0,1)^k$ and $a_l^{(n)}\in A$, $n\in\NN$, $l\in\{1,\ldots,k+1\}$, be such that $\lim_{n\to\infty}\frac{a_l^{(n)}}{a_{k+1}^{(n)}}=r_l$ for each $l\in\{1,\ldots,k\}$. For sufficiently large $n\in\NN$ we have $a_1^{(n)}<\cdots<a_k^{(n)}<a_{k+1}^{(n)}$. If for infinitely many $n$ we have $a_1^{(n)},\ldots ,a_k^{(n)}, a_{k+1}^{(n)}\in B_m$ for some $m\in\NN$, then
$$r_l=\lim_{n\to\infty}\frac{a_l^{(n)}}{a_{k+1}^{(n)}}\geq 1-\frac{d}{2},\ l\in\{1,\ldots ,k\}.$$
Otherwise, by the assumption on enumeration of the $k-1$-tuples $(q_1^{(n)},\ldots,q_{k-1}^{(n)})$ and from the definition of the sets $D_n$ we have
$$r_1=\lim_{n\to\infty}\frac{a_1^{(n)}}{a_{k+1}^{(n)}}\leq 1-d.$$ This means that all the accumulation points of $R^{k+1}(A)\cap (0,1)^k$ have all the coordinates at least equal to $1-\frac{d}{2}$ or at least one coordinate at most equal to $1-d$. This means that $R^{k+1}(A)$ is not dense in $(0,1)^k$.

At this moment we compute $\underline{D}(A)$. Write $A=\{a_1<a_2<a_3<\ldots \}$ and set $a_0=0$. Consider several cases for the value of $\frac{1}{a_t}\max_{1\leq j\leq t}(a_j-a_{j-1})$.
\begin{itemize}
\item If $a_t\in B_n$ for sufficiently large $n\in\NN$, then
$$\frac{1}{a_t}\max_{1\leq j\leq j_n}(a_j-a_{j-1})>\frac{d\cdot (4n-3)!-(4n-4)!}{\frac{2d}{2-d}\cdot (4n-3)!}=1-\frac{d}{2}-\frac{2-d}{2d(4n-3)}\to 1-\frac{d}{2}>d\ \textrm{as}\ n\to\infty.$$
\item If $a_t=(4n-1)!\cdot q_l^{(n)}$, $n\geq 2$, $l\in\{1,\ldots,k-1\}$, then
\begin{align*}
&\frac{1}{a_t}\max_{1\leq j\leq j_n+l}(a_j-a_{j-1})\geq\frac{1}{a_t}\max_{1\leq s\leq l}(a_{t-l+s}-a_{t-l+s-1})\geq\frac{1}{a_t}\frac{a_t-a_{t-l}}{l}\\
&>\frac{1}{(4n-1)!\cdot q_l^{(n)}}\frac{(4n-1)!\cdot q_l^{(n)}-(4n-3)!}{l}=\frac{1-\frac{1}{(4n-2)(4n-1)\cdot q_l^{(n)}}}{l}\to\frac{1}{l}>\frac{1}{k}\ \textrm{as}\ n\to\infty,
\end{align*}
where the convergence $\frac{1}{(4n-2)(4n-1)\cdot q_l^{(n)}}\to 0$ in the second line follows from the assumption that $(4n-1)!\cdot q_l^{(n)}\geq (4n-2)!$.
\item If $a_t=(4n-1)!$, $n\in\NN$, then
\begin{align*}
&\frac{1}{a_t}\max_{1\leq j\leq j_n+l}(a_j-a_{j-1})\geq\frac{1}{a_t}\max_{1\leq s\leq k}(a_{t-k+s}-a_{t-k+s-1})\geq\frac{1}{a_t}\frac{a_t-a_{t-k}}{k}\\
&>\frac{1}{(4n-1)!}\frac{(4n-1)!-(4n-3)!}{k}=\frac{1-\frac{1}{(4n-2)(4n-1)}}{k}\to\frac{1}{k}\ \textrm{as}\ n\to\infty.
\end{align*}
\item If $a_t=\left\lceil\frac{(4n-1)!}{(1-d)^j}\right\rceil\in D_n$, $j,n\in\NN$, then
\begin{align*}
&\frac{1}{a_t}\max_{1\leq j\leq t}(a_j-a_{j-1})\geq\frac{1}{a_t}(a_t-a_{t-1})\\
&=\frac{1}{\left\lceil\frac{(4n-1)!}{(1-d)^j}\right\rceil}\left(\left\lceil\frac{(4n-1)!}{(1-d)^j}\right\rceil-\left\lceil\frac{(4n-1)!}{(1-d)^{j-1}}\right\rceil\right)\to 1-(1-d)=d\ \textrm{as}\ n\to\infty.
\end{align*}
Moreover, if $a_t=\max D_n$, then $j=\left\lfloor -\frac{\log 4n}{\log (1-d)}\right\rfloor\to\infty$ as $n\to\infty$. As a result, we have
\begin{align*}
&a_t-a_{t-1}=\left\lceil\frac{(4n-1)!}{(1-d)^j}\right\rceil-\left\lceil\frac{(4n-1)!}{(1-d)^{j-1}}\right\rceil >\frac{(4n-1)!}{(1-d)^j}-\frac{(4n-1)!}{(1-d)^{j-1}}-1=\frac{d}{(1-d)^j}(4n-1)!-1\\
&>\max\left\{(4n-1)!,\left\lceil\frac{(4n-1)!}{(1-d)^s}\right\rceil-\left\lceil\frac{(4n-1)!}{(1-d)^{s-1}}\right\rceil :\ s\in\{1,\ldots ,j-1\}\right\}
\end{align*}
for sufficiently large $n$. Hence, $\frac{1}{a_t}\max_{1\leq j\leq t}(a_j-a_{j-1})=\frac{1}{a_t}(a_t-a_{t-1})\to d$, $n\to\infty$.
\end{itemize}
Summing up, $\underline{D}(A)=\liminf_{t\to\infty}\frac{1}{a_t}\max_{1\leq j\leq t}(a_j-a_{j-1})=d$.

We are left with the computation of $\lim_{t\to\infty}\frac{log A(t)}{\log t}$. We do this by estimating the lower and upper limit from below and above, respectively. On one hand,
\begin{align*}
&\liminf_{t\to\infty}\frac{\log A(t)}{\log t}=\lim_{n\to\infty}\frac{\log A(d\cdot (4n+1)!)}{\log (d\cdot (4n+1)!)}\geq\lim_{n\to\infty}\frac{\log(\# B_n)}{\log (d\cdot (4n+1)!)}\\
&=\lim_{n\to\infty}\frac{\log\left(\left(\frac{2d}{2-d}\cdot (4n-3)!\right)^\lambda -(d\cdot (4n-3)!)^\lambda\right)}{\log d+(4n+1)\log(4n+1)}\\
&=\lim_{n\to\infty}\frac{\log\left(\left(\frac{2d}{2-d}\right)^\lambda -d^\lambda\right)+\lambda\cdot (4n-3)\log(4n-3)}{\log d+(4n+1)\log(4n+1)}=\lambda,
\end{align*}
where we used the fact that $\frac{\log n!}{n\log n}\to 1$ as $n\to\infty$. On the other hand,
\begin{align*}
&\limsup_{t\to\infty}\frac{\log A(t)}{\log t}=\lim_{n\to\infty}\frac{\log A\left(\frac{2d}{2-d}\cdot (4n-3)!\right)}{\log \left(\frac{2d}{2-d}\cdot (4n-3)!\right)}\leq\lim_{n\to\infty}\frac{\log\left(\frac{2d}{2-d}\cdot (4n-3)!\right)^\lambda}{\log\left(\frac{2d}{2-d}\cdot (4n-3)!\right)}=\lambda,
\end{align*}
where we used the fact that
\begin{align*}
\# B_n<&\#\left(\left\{\left\lceil j^{\frac{1}{\lambda}}\right\rceil :\ j\in\NN\right\}\cap((4n-4)!,(4n-3)!]\right),\\
\# C_n<&\#\left(\left\{\left\lceil j^{\frac{1}{\lambda}}\right\rceil :\ j\in\NN\right\}\cap((4n-3)!,(4n-1)!]\right),\\
\# D_n<&\#\left(\left\{\left\lceil j^{\frac{1}{\lambda}}\right\rceil :\ j\in\NN\right\}\cap((4n-1)!,(4n)!]\right)
\end{align*}
for sufficiently large $n$. Thus we showed that $\lim_{t\to\infty}\frac{log A(t)}{\log t}=\lambda$.

Now we consider the case of $k=1$ and $\lambda\in (0,1]$. Define
$$d_n=\begin{cases}
d&\mbox{ if }\ d\in (0,1),\\
\left(1-\frac{1}{n}\right)^{\frac{1}{\lambda}}&\mbox{ if }\ d=1,
\end{cases}\quad n\in\NN,$$
and
$$A_n=\left\{\lceil d_n\cdot n!\rceil , n!, \left\lceil j^{\frac{1}{\lambda}}\right\rceil :j\in\NN\cap [(d_n\cdot n!)^{\lambda},(n!)^{\lambda}]\right\}.$$
Then we set
$$A=\bigcup_{n\in\NN}A_n.$$
We see that $R(A)$ is not dense in $(0,1)$. Indeed, if $a<b$ are two elements of $A$ and $a,b\in A_n$ for some $n\in\NN$, then $\frac{a}{b}\geq d_n$. If $a\in A_m$ and $b\in A_n$ for $m<n$, then $\frac{a}{b}\leq\frac{1}{nd_n}$.

Now, we compute the dispersion of $A$. We begin with a remark that among all the elements of $A$ not exceeding $n!$ the greatest difference between two consecutive elements is $\lceil d_n\cdot n!\rceil-(n-1)!$ for sufficiently large $n\in\NN$. Indeed, we have
$$\lceil d_n\cdot n!\rceil-(n-1)!\geq d_n\cdot n!-(n-1)!=(d_n\cdot n-1)(n-1)!>(n-1)!$$
for $n>\frac{2}{d}$ if $d\in (0,1)$ and $n>2^{1+\frac{1}{\lambda}}$ if $d=1$. Moreover,
\begin{align*}
&\left\lceil j^{\frac{1}{\lambda}}\right\rceil-\left\lceil (j-1)^{\frac{1}{\lambda}}\right\rceil<j^{\frac{1}{\lambda}}+1-(j-1)^{\frac{1}{\lambda}}\leq n!+1-\left((n!)^{\lambda}-1\right)^{\frac{1}{\lambda}}=n!(1-(1-(n!)^{-\lambda})^{\frac{1}{\lambda}})+1\\
&\leq n!\left(1-\left(1-\frac{1}{\lambda}(n!)^{-\lambda}\right)\right)+1=n!\cdot \frac{1}{\lambda}(n!)^{-\lambda}+1=\frac{1}{\lambda}(n!)^{1-\lambda}+1<(n-1)!
\end{align*}
for sufficiently large $n\in\NN$, where the second inequality in the first line follows from the assumption that $j\leq (n!)^{\lambda}$, the first inequality in the second line comes from Bernoulli inequality and the last inequality holds for $n\gg 0$ as $\lambda>0$. As a result,
$$\underline{D}(A)=\lim_{n\to\infty}\frac{d_n\cdot n!-(n-1)!}{n!}=\lim_{n\to\infty}\left(d_n-\frac{1}{n}\right)=d.$$

We compute $\lim_{t\to\infty}\frac{\log A(t)}{\log t}$ by estimating the lower and upper limits from below and above, respectively.
\begin{align*}
&\liminf_{t\to\infty}\frac{\log A(t)}{\log t}=\lim_{n\to\infty}\frac{\log A(d_{n+1}\cdot (n+1)!)}{\log (d_{n+1}\cdot (n+1)!)}\geq\lim_{n\to\infty}\frac{\log (\# A_n)}{\log d_{n+1} + \log (n+1)!}\\
&=\lim_{n\to\infty}\frac{\log ((n!)^{\lambda}-(d_n\cdot n!)^{\lambda})}{\log d_{n+1} + (n+1)\log (n+1)}=\lim_{n\to\infty}\frac{\log(1-(d_n)^{\lambda})+\lambda\log n!}{\log d_{n+1} + (n+1)\log (n+1)}\\
&=\lim_{n\to\infty}\frac{\log(1-(d_n)^{\lambda})+\lambda\cdot n\log n}{\log d_{n+1} + (n+1)\log (n+1)}=\lambda.
\end{align*}

Next, we estimate the upper limit from above with the use of the fact that $A\subset\left\{\lceil j\rceil^{\frac{1}{\lambda}}:\ j\in\NN\right\}$.
\begin{align*}
\limsup_{t\to\infty}\frac{\log A(t)}{\log t}=\lim_{n\to\infty}\frac{\log A(n!)}{\log n!}\leq\lim_{n\to\infty}\frac{\log (n!)^\lambda}{\log n!}=\lambda
\end{align*}
Summing up our estimations, we get
$$\lambda\leq\liminf_{t\to\infty}\frac{\log A(t)}{\log t}\leq\limsup_{t\to\infty}\frac{\log A(t)}{\log t}\leq\lambda,$$
which means that $\lim_{t\to\infty}\frac{\log A(t)}{\log t}$ exists and is equal to $\lambda$.

We are left with the case of $k=1$ and $\lambda=0$. It is quite easy as we may take
$$A=\begin{cases}
\{\lceil (1-d)^{-k}\rceil:k\in\NN\}&\mbox{ for }d\in (0,1),\\
\{k!:k\in\NN\}&\mbox{ for }d=1.\\
\end{cases}$$
as a required set. The facts that $R(A)$ is not dense in $\RR_+$, $\underline{D}(A)=d$ and $\lim_{t\to\infty}\frac{\log A(t)}{\log t}=0$ are classical ones, so we leave their proof to the reader.
\end{proof}

\begin{Remark}
{\rm Replacing some of the sets $B_n$ or $A_n$ by $$\left\{\left\lceil j^{\frac{1}{\kappa}}\right\rceil :\ j\in\NN\right\}\cap \left(\max\{(4n-4)!,d\cdot (4n-3)!\},\frac{2d}{2-d}\cdot (4n-3)!\right],$$ where $\kappa<\lambda$ (in case of $\kappa=0$ we only delete some sets $B_n$ or $A_n$), we can even construct a set $A\subset\NN$ with $R^k(A)$ dense in $\RR_+^{k-1}$, $R^{k+1}(A)$ not dense in $\RR_+^k$, dispersion equal to $d\in \left(0,\frac{1}{k}\right]$, and $\liminf_{t\to\infty}\frac{\log A(t)}{\log t}=\kappa$ and $\limsup_{t\to\infty}\frac{\log A(t)}{\log t}=\lambda$. We decided not to consider distinct values of $\liminf_{t\to\infty}\frac{\log A(t)}{\log t}$ and $\limsup_{t\to\infty}\frac{\log A(t)}{\log t}$ in Theorem \ref{arbitrarydispersion} as it would complicate its proof that is tricky even in the present form.}
\end{Remark}

\begin{Remark}\label{conexpnotimpRVq}
{\rm With the use of Theorem \ref{arbitrarydispersion} we are able to conclude that the condition \linebreak$\lim_{n\to\infty}\frac{\log A(n)}{\log n}=q$ for arbitrarily fixed $q\in (0,1]$ does not imply that $A(x)\in\RV_q$. By Theorem \ref{arbitrarydispersion} we know that there exists a subset $A$ of $\NN$ such that $\lim_{n\to\infty}\frac{\log A(n)}{\log n}=q$ and $\underline{D}(A)>0$. Hence, $A$ is not $(N)$-dense and, as a result, $A\not\in\RV_q$.}
\end{Remark}

Theorem \ref{arbitrarydispersion} concerns the case of positive dispersion. Our last theorem compliments it as it shows that there are $(N)$-dense sets $A\subset\NN$ (and hence the sets with dispersion zero) with arbitrary value of $\lim_{n\to\infty}\frac{\log A(n)}{\log n}$, and hence the exponent of convergence.

\begin{Theorem}\label{convexpdisp}
For each $\lambda\in [0,1]$ there exists an $(N)$-dense set $A\subset\NN$ such that $\lim_{n\to\infty}\frac{\log A(n)}{\log n}=\lambda$ (in particular $A$ has convergence exponent equal to $\lambda$).
\end{Theorem}

\begin{proof}
If $\lambda>0$, then we put
$$A=\left\{\lceil j\rceil^{\frac{1}{\lambda}}:\ j\in\NN\right\}.$$
Then, of course,
$$\lim_{t\to\infty}\frac{\log A(t)}{\log t}=\lim_{t\to\infty}\frac{\log t^\lambda}{\log t}=\lambda$$
and
$$\lim_{j\to\infty}\frac{\left\lceil (j+1)^{\frac{1}{\lambda}}\right\rceil}{\left\lceil j^{\frac{1}{\lambda}}\right\rceil}=\lim_{j\to\infty}\left(\frac{j+1}{j}\right)^{\frac{1}{\lambda}}=1,$$
so by Theorem \ref{Ndense} we know that $A$ is $(N)$-dense.

We are left with the case of $\lambda=0$. Define 
$$A_n=\{j^n:j\in\{n^{n-1},n^{n-1}+1,\ldots,(n+1)^{n+1}-1,(n+1)^{n+1}\}\},\quad n\in\NN$$ and $$A=\bigcup_{n\in\NN}A_n.$$
Note that $\max A_{n}=\min A_{n+1}=(n+1)^{n(n+1)}$.

We check that $\sum_{a\in A}\frac{1}{a^{\alpha}}<\infty$ for each $\alpha>0$.
$$\sum_{a\in A}\frac{1}{a^{\alpha}}=\sum_{n=1}^{\infty}\sum_{j=n^{n-1}}^{(n+1)^{n+1}}\frac{1}{j^{n\alpha}}<\sum_{n=1}^{\infty}\frac{(n+1)^{n+1}}{n^{(n-1)n\alpha}}<\infty,$$
where the first inequality follows from the facts that
$$\#\{n^{n-1},n^{n-1},\ldots,(n+1)^{n+1}-1,(n+1)^{n+1}\}<(n+1)^{n+1}$$
and 
$$\frac{1}{j^n}\leq\frac{1}{n^{(n-1)n}},\quad j\in\{n^{n-1},n^{n-1},\ldots,(n+1)^{n+1}-1,(n+1)^{n+1}\}.$$ 
The convergence for the last series comes from the root test as
$$0<\sqrt[n]{\frac{(n+1)^{n+1}}{n^{(n-1)n\alpha}}}=\frac{(n+1)^{1+\frac{1}{n}}}{n^{(n-1)\alpha}}<\frac{n^4}{n^{(n-1)\alpha}}\to 0\ \textrm{as}\ n\to\infty,$$
where the last inequality holds for $n\geq 2$ because then we have $1+\frac{1}{n}<2$ and $n+1<n^2$.

Now we shall prove that $\lim_{t\to\infty}\frac{a_{t+1}}{a_t}=1$, where $A=\{a_1<a_2<a_3<\ldots\}$. Then, by Theorem \ref{Ndense} we will know that $A$ is $(N)$-dense. Let us take two consecutive elements $a_t, a_{t+1}\in A$. Then, there exists a $n=n(t)\in\NN$ such that $a_t, a_{t+1}\in A_k$. This means that $a_t=j^n$ and $a_{t+1}=(j+1)^n$ for some $j\in\{n^{n-1},n^{n-1},\ldots,(n+1)^{n+1}-1,(n+1)^{n+1}\}$. Then,
$$1<\frac{a_{t+1}}{a_t}=\frac{(j+1)^n}{j^n}=\left(1+\frac{1}{j}\right)^n\leq\left(1+\frac{1}{n^{n-1}}\right)^n=\left(\left(1+\frac{1}{n^{n-1}}\right)^{n^{n-1}}\right)^{n^{2-n}}\to 1\ \textrm{as}\ n\to\infty,$$
where the inequality follows from the assumption $j\geq n^{n-1}$ and the convergence holds because $\left(1+\frac{1}{n^{n-1}}\right)^{n^{n-1}}\to e$ and $n^{2-n}\to 0$ as $n\to\infty$. Thus, $\lim_{t\to\infty}\frac{a_{t+1}}{a_t}=1$, which ends the proof.
\end{proof}

\begin{Remark}\label{convexpNdense}
{\rm Intersecting the sets $A$ constructed in the above proof with the set $\bigcup_{k=1}^{\infty}[(2k)!,(2k+1)!]$ we see that for each $\lambda\in [0,1]$ there exists a set $B\in\NN$ such that $\lim_{t\to\infty}\frac{\log B(t)}{\log t}=\lambda$, $\underline{D}(B)=0$ but $B$ is not $(N)$-dense.}
\end{Remark}

\begin{Remark}
{\rm The set $A$ from Example \ref{nolim} in the case of $p>0$ is $(N)$-dense as the quotient of two consecutive elements of $A$ tends to $1$. Hence there exists an $(N)$-dense set $A\subset\NN$ such that $\liminf_{t\to\infty}\frac{\log A(t)}{\log t}=p$ and $\limsup_{t\to\infty}\frac{\log A(t)}{\log t}=q$, where $0<p<q\leq 1$ are arbitrary. Meanwhile the set
$A'=\bigcup_{k=1}^{\infty}(A_{4k-1}\cup A_{4k}),$
where $A_k$ are as in the example, has dispersion equal to $0$ but is not $(N)$-dense, and $\liminf_{t\to\infty}\frac{\log A'(t)}{\log t}=p$ and $\limsup_{t\to\infty}\frac{\log A'(t)}{\log t}=q$.

It is also possible to construct an $(N)$-dense set $B\in\NN$ such that $\liminf_{t\to\infty}\frac{\log B(t)}{\log t}=0$ and $\lim_{t\to\infty}\frac{\log B(t)}{\log t}=q$, where $q\in (0,1]$ is arbitrary. Namely,
$$B=A\cup\bigcup_{k=1}^{\infty}\left\{\left\lceil j^{\frac{1}{q}}\right\rceil :\ b_k\leq j^{\frac{1}{q}}\leq c_k\right\},$$
where $A$ is the set from the proof of Theorem \ref{convexpdisp} in the case of $\lambda=0$ and $(b_k)_{k\in\NN}$, $(c_k)_{k\in\NN}$ are appropriately chosen sequences increasing to infinity. Of course, the set
$$B'=\bigcup_{k=1}^{\infty}\left\{\left\lceil j^{\frac{1}{q}}\right\rceil :\ b_k\leq j^{\frac{1}{q}}\leq c_k\right\}$$
has dispersion $0$, is not $(N)$-dense, and $\liminf_{t\to\infty}\frac{\log B'(t)}{\log t}=0$ and $\lim_{t\to\infty}\frac{\log B'(t)}{\log t}=q$.}
\end{Remark}

\begin{Remark}\label{Ndensenotimpliminf}
{\rm The set $A$ constructed for $\lambda=d=0$ in the proof of Theorem \ref{convexpdisp} shows that the condition $\forall c>1:\liminf_{t\to\infty}\frac{A(ct)}{A(t)}>1$ is essentially stronger than $\forall c>1:A(ct)>A(t)$ for $t\gg 0$. Indeed, let $c>1$ be arbitrary. Let $t\in\RR_+$. Then, there exists a unique $k=k(t)$ such that $t\in [k^{k-1},(k+1)^{k+1})$. Thus, we have
$$A(t)\geq\#\{j\in\NN:j^k\leq t\}\leq\sqrt[k]{t}$$
and
$$A(ct)-A(t)\leq\#\{j\in\NN:j^k\in (t,ct]\}=\#(\NN\cap (\sqrt[k]{t},\sqrt[k]{ct}])\leq\sqrt[k]{ct}-\sqrt[k]{t}.$$
Consequently,
$$\lim_{t\to\infty}\frac{A(ct)}{A(t)}=\lim_{t\to\infty}\Big(1+\frac{A(ct)-A(t)}{A(t)}\Big)\leq\lim_{t\to\infty}\Big(1+\frac{\sqrt[k]{ct}-\sqrt[k]{t}}{\sqrt[k]{t}}\Big)=\lim_{t\to\infty}\sqrt[k(t)]{c}=1,$$
as $\lim_{t\to\infty}k(t)=\infty$. On the other hand, $A$ is $(N)$-dense, as we noted in Remark \ref{convexpNdense}. Hence, by Theorem \ref{Ndense} we know that $A(ct)>A(t)$ for each $c>1$ and $t\gg 0$.

Moreover, for any nonempty set $B\subset A$ and number $c>1$ we have $\liminf_{t\to\infty}\frac{B(ct)}{B(t)}=1$. Assume the contrary. Let $B\subset A$, $c,g,h,t_0>1$ be such that $g>h$, $B(t_0)>0$, $\frac{B(ct)}{B(t)}\geq g$ and $\frac{A(ct)}{A(t)}\leq h$ for $t\geq t_0$. Then, $$g^nB(t_0)\leq B(c^nt_0)\leq A(c^nt_0)\leq h^nA(t_0), \quad n\in\NN.$$ On the other hand, $g^nB(t_0)>h^nA(t_0)$ for sufficiently large $n\in\NN$ as $g>h$ and $B(t_0)>0$ - a contradiction. Thus, $\liminf_{t\to\infty}\frac{B(ct)}{B(t)}=1$ for each $\emptyset\neq B\subset A$ and $q>1$.}
\end{Remark}

\section{Summary of results}

Firstly, we show the relationships between the implications (A16), (A17), (A18) and the properties (statements) introduced or proved in this paper.
\\Further let $A=\{a_1<a_2<\cdots\}\subset\NN$. Then we have first chain of implications as $\lambda(A)>0$.

\begin{align*}
& A\ \textrm{is regul. seq. with exp.}\ 0<q\leq 1\ \substack{{(A18)}\\ \imp\\ \Longleftarrow\\ {(A14)-i)}}\ 
\lim_{n\to\infty} \frac1{na_n} \sum_{i=1}^{n} {a_i}=\frac q{q+1}\ {\eq_{(A6)}} \ A\in\Uq\ \\
& {\eq_{Corol. \ref{C2}}}\ A(x)\in\RV_q\ \substack{{(A19)}\\ \imp\\ \not\Longleftarrow\\ Rem. \ref{conexpnotimpRVq}}\ \lim_{n \to \infty}\frac{\log A(n)}{\log n}=q\ {\eq_{(A15)}}\ \lim_{n \to \infty}\frac{\log n}{\log a_n}=q\ \substack{{trivially}\\ \imp\\ \not\Longleftarrow\\ Ex. \ref{nolim}}\ \lambda(A)=q \\ 
& {\eq_{def.}} \ A\in\Ieq\setminus\Iq\ \substack{{(A10)}\\ \imp\\ \not\Longleftarrow\\ trivially}\ A\in\Ixq.\ 
\end{align*}

On the other hand we have second chain of implications, where there is $\lambda(A)=0$.

\begin{align*}
& \lim_{n\to\infty} \frac1{na_n} \sum_{i=1}^{n} {a_i}=0\ {\eq_{(A4)}}\ 
A\in\Uc\ {\eq_{Corol. \ref{C2}}}\ A(x)\in\RV_0\ \substack{{Corol. \ref{C4}}\\ \imp\\ \not\Longleftarrow\\ Rem. \ref{liminfnotimpRVq}}\ \lambda(A)=0 \\
& {\eq_{def.}} \ A\in\I0\ {\eq_{(A10)}}\ A\in\Ic.\ 
\end{align*}

The next chain of statements is continuation of the first one and contains e.g. extension of (A20), equivalent conditions for $(N)$-denseness and characterization of condition $\underline D(A)=0$. Every condition but the last three in the first from above chains implies $(N)$-denseness of $A$.

\begin{align*}
& \exists q\in (0,1]: A(x)\in\RV_q\ {\eq_{Corol. \ref{C3}}}\ \forall\  c>1:\ \lim_{t\to\infty}\frac{A(ct)}{A(t)}>1\ \substack{{trivially}\\ \imp\\ \not\Longleftarrow\\ Rem. \ref{liminfnotimplim}}\ \forall\  c>1:\ \liminf_{t\to\infty}\frac{A(ct)}{A(t)}>1\ \\
& \substack{{trivially}\\ \imp\\ \not\Longleftarrow\\ Rem. \ref{Ndensenotimpliminf}}\ \forall\  c>1:\ A(ct)>A(t)\ \textrm{for}\ t\gg 0\ \\
& {\eq_{Thm \ref{Ndense}}}\ \forall\ k\geq 2\ \forall\ B\subset\NN\ \textrm{infinite:}\ R^k(A;B)\ \textrm{is dense in}\ \RR_+^{k-1}\ \\
& {\eq_{Thm \ref{Ndense}}}\ \forall\ k\geq 2\ \forall\ B\subset\NN\ \textrm{infinite:}\ D^k(A;B)\ \textrm{is dense in}\ \SS_+^{k-1}\ \\
& {\eq_{Thm \ref{Ndense}}}\ \exists\ k\geq 2\ \forall\ B\subset\NN\ \textrm{infinite:}\ R^k(A;B)\ \textrm{is dense in}\ \RR_+^{k-1}\ \\
& {\eq_{Thm \ref{Ndense}}}\ \exists\ k\geq 2\ \forall\ B\subset\NN\ \textrm{infinite:}\ D^k(A;B)\ \textrm{is dense in}\ \SS_+^{k-1}\ \\
& {\eq_{Thm \ref{Ndense}}}\ \forall\ B\subset\NN\ \textrm{infinite}\ \exists\ k\geq 2:\ R^k(A;B)\ \textrm{is dense in}\ \RR_+^{k-1}\ \\
& {\eq_{Thm \ref{Ndense}}}\ \forall\ B\subset\NN\ \textrm{infinite}\ \exists\ k\geq 2:\ D^k(A;B)\ \textrm{is dense in}\ \SS_+^{k-1}\ \\
& {\eq_{Thm \ref{Ndense}}}\ A\ \textrm{is}\ (N)-\textrm{dense}\ {\eq_{Thm \ref{Ndense}}}\ \lim_{n \to \infty}\frac{a_{n+1}}{a_n}=1\ \\
& {\eq_{Thm \ref{Ndense}}}\ \lim_{n \to \infty}\frac{1}{a_n}\max\{a_1,a_{i+1}-a_i: i\in\{1,\ldots,n-1\}\}=0 \substack{{trivially}\\ \imp\\ \not\Longleftarrow\\ Rem. \ref{D0notimpNdense}}\ \underline D (A)=0\ \\
& {\eq_{Thm \ref{D0}}}\ \forall\ k\geq 2:\ R^k(A)\ \textrm{is dense in}\ \RR^{k-1}_+\ {\eq_{Lem. \ref{directionratio}}}\ \forall\ k\geq 2:\ D^k(A)\ \textrm{is dense in}\ S^{k-1}_+\ \\
& \substack{{trivially}\\ \imp\\ \not\Longleftarrow\\ Thm \ref{arbitrarydispersion}}\ \exists\ k\geq 2:\ D^k(A)\ \textrm{is dense in}\ S^{k-1}_+\ {\eq_{Lem. \ref{directionratio}}}\ \exists\ k\geq 2:\ R^k(A)\ \textrm{is dense in}\ \RR^{k-1}_+\ \\
& \substack{{see\ proof\ of\ vi)\imp viii)\ in\ Thm 3}\\ \imp\\ \not\Longleftarrow\\ Thm \ref{arbitrarydispersion}}\ A\ \textrm{is}\ (R)\textrm{-dense}.\
\end{align*}

The last chain of implications complements the third one as it gives a connection between denseness of sets $R^l(A)$ in $\RR_+^{l-1}$ and $R^k(A)$ in $\RR_+^{k-1}$ for $2\leq k\leq l$ and a necessary condition for denseness of $R^k(A)$ in $\RR_+^{k-1}$ in terms of dispersion (in fact, this necessary condition is a natural generalization of (A22)).

\begin{align*}
R^l(A)\ \textrm{is dense in}\ \RR^{l-1}_+\ \substack{{Lem. \ref{directionratio}},\ (A23)\\ \imp\\ \not\Longleftarrow\\ (A23)}\ R^k(A)\ \textrm{is dense in}\ \RR^{k-1}_+\ \substack{{Thm \ref{1/k}}\\ \imp\\ \not\Longleftarrow\\ Thm \ref{arbitrarydispersion}}\ \underline{D}(A)\leq\frac{1}{k}.\ 
\end{align*}

The implications in the above chains and these ones derived from them by transitivity law are the only valid implications between conditions and their conjunctions presented in the chains.

\end{document}